\documentclass{amsart}

\usepackage{amssymb,amsmath,amscd,graphicx,epstopdf,amsthm}

\usepackage{color}

\newtheorem{theorem}{Theorem}
\newtheorem{lemma}[theorem]{Lemma}
\newtheorem{proposition}[theorem]{Proposition}
\newtheorem{corollary}[theorem]{Corollary}

\newtheorem{remark}[theorem]{Remark}

\numberwithin{equation}{section}
\numberwithin{figure}{section}
\numberwithin{theorem}{section}

\def\CC{{\mathcal E}}
\def\EE{{\mathcal E}}

\def\FF{{\mathcal F}}

\def\GG{{\mathfrak G}}

\def\P{{\mathcal P}}

\def\R{{\mathbb R}}
\def\RR{{\mathcal R}}

\def\Z{{\mathbb Z}}

\def\one{\mathbf 1}

\def\wN{{\widetilde{N}}}

\def\wX{{\bar X}}
\def\hM{{\widehat{M}}}
\def\hN{{\widehat{N}}}

\def\Hom{\operatorname{Hom}}

\def\Net{\operatorname{Net}}

\def\Poi{{\{\cdot,\cdot\}}}

\def\Rat{\operatorname{Rat}}

\def\diag{\operatorname{diag}}
\def\dim{\operatorname{dim}}

\def\id{\operatorname{id}}
\def\ind{\operatorname{ind}}

\def\s{\operatorname{sign}}

\def\val{\operatorname{val}}

\def\:{{:\ }}

\begin{document}

\title{ Poisson Geometry of Directed Networks in an Annulus}

\author{Michael Gekhtman}
%    Address of record for the research reported here
\address{Department of Mathematics, University of Notre Dame, Notre Dame,
IN 46556}
\email{mgekhtma@nd.edu}
%    \thanks will become a 1st page footnote.
%\thanks{The first author was supported in part by NSF Grant DMS \#0400484.}

%    Information for second author
\author{Michael Shapiro}
\address{Department of Mathematics, Michigan State University, East Lansing,
MI 48823}
\email{mshapiro@math.msu.edu}
%    \thanks will become a 1st page footnote.
%\thanks{The second author was supported in part by NSF Grants DMS \#0401178 and PHY\#0555346.}

%    Information for third author
\author{Alek Vainshtein}
\address{Department of Mathematics AND Department of Computer Science, University of Haifa, Haifa,
Mount Carmel 31905, Israel}
\email{alek@cs.haifa.ac.il}
%    \thanks will become a 1st page footnote.
%\thanks{The third author was supported in part by NSF Grant \#000000.}

\date\today
%\keywords{AMS-\LaTeX}

\begin{abstract} 
As a generalization of Postnikov's construction [P], we
define a map
from the space of edge weights of a directed network in an annulus
into a space
of loops in the Grassmannian. We then show that universal Poisson
brackets introduced
for the space of edge weights in [GSV3] induce a family of Poisson
structures on rational-valued
matrix functions and on the space
of loops in the Grassmannian. In the former case, this family
includes, for a particular kind of networks,
the Poisson bracket associated with the trigonometric R-matrix.
\end{abstract}

\maketitle

\bigskip 

\section{Introduction}

This is the second  in the series of four papers initiated 
by \cite{GSV3} and devoted to geometry behind directed networks on surfaces, with 
a particular emphasis on their Poisson properties.

In \cite{GSV3}, we concentrated on  Postnikov's construction \cite{Postnikov} that uses
weighted directed planar graphs to parametrize cells in Grassmannians. We found
that the space of edge weights of networks in a disk can be endowed with a natural family of Poisson brackets
(that we called {\em universal\/}) that "respects" the operation of concatenation of diagrams.
We have shown that, under Postnikov's parametrization, these Poisson brackets induce a two-parameter family of Poisson brackets on the Grassmannian. Every Poisson bracket in this family is compatible (in the sense of \cite{GSV1,GSV2}) with the cluster algebra on the Grassmannian described in \cite{GSV1,Scott} and, on the other hand, endows
the Grassmannian with a structure of a Poisson homogeneous space with respect
to the natural action of the general linear group equipped with an R-matrix Poisson-Lie
structure.

As was announced in \cite{GSV3}, the current paper builds a parallel theory for directed weighted networks in an annulus 
(or, equivalently, on a cylinder). First, we have to modify the definition of the boundary measurement map, whose
image now consists of rational valued matrix functions of an auxiliary parameter $\lambda$ associated with
the notion of a {\em cut\/} (see Section 1). We then show that the analogue of Postnikov's construction leads
to a map into the space of loops in the Grassmannian. Universal Poisson brackets for networks in an annulus are defined 
in exactly the same way as in the case of a disk. We show that they induce  a two-parameter family of Poisson brackets  
on rational-valued boundary measurement matrices. In
particular, when sources and sinks belong to distinct circles bounding the annulus, one of the generators 
of these family coincides with the Sklyanin R-matrix bracket associated with the trigonometric solution of the 
classical Yang-Baxter equation in $sl(n)$. Moreover, we prove that the two-parameter family of Poisson brackets can be 
further pushed-forward to the space of loops in the Grassmannian. In proving
the latter, we departed from the approach of \cite{GSV3} where the similar result was obtained via a more or less 
straightforward calculation. Such an approach would have been too cumbersome in our current setting. Instead, we found a 
way to utilize so-called {\em face weights\/} and  their behavior
under {\em path-reversal maps}. 

The forthcoming third paper in this series \cite{GSV4} focuses on  particular graphs in an annulus that can be used  to 
introduce a cluster algebra structure on the coordinate ring of the space
of normalized rational functions in one variable. This space is birationally equivalent, via the Moser map \cite{moser},  
to any minimal irreducible coadjoint orbit of the group of upper triangular  matrices associated with a Coxeter element 
of the permutation group. In this case, the Poisson bracket compatible with the cluster algebra structure coincides with 
the quadratic Poisson bracket studied in \cite{FayGekh1, FayGekh2} in the context of Toda flows on minimal orbits. 
We show that cluster transformations serve as B\"acklund-Darboux transformations between different minimal Toda flows. 
The fourth paper \cite{GSV5} solves, in the case of graphs in an annulus with one
source and one sink, the inverse problem of restoring the weights from the image
of the generalized Postnikov map. In the case of arbitrary planar graphs in a disk, this
problem was completely solved by Postnikov \cite{Postnikov} who proved that for a fixed minimal graph, the space of 
weights modulo gauge action is birational to its image. 
To the contrary, already for simplest graphs in an annulus, the corresponding map  can only be shown to be finite.

The original application of directed weighted planar networks was in the study of total positivity, 
both in $GL_n$ \cite{KarlinMacGregor, Brenti, BFZ, FZ_Intel, Fallat} and in Grassmannians \cite{Postnikov}.
We do not address this issue for networks in an annulus. It has been studied in a recent preprint
\cite{Pavlo}.

The paper is organized as follows.

In Section \ref{netsBM2}, we introduce a notion of a perfect  network in an annulus and
associate with every such network a {\em matrix of boundary measurements}. 
Each boundary measurement is shown to be a rational function in edge weights   and in an auxiliary parameter $\lambda$, see Corollary~\ref{sfree3}. Besides, we define the space
of face and trail weights, a generalization of the space of face weights studied in \cite{GSV3} for the case of networks in a disk, and provide its cohomological interpretation, see Section~\ref{cohomol}. 

In  Section \ref{PSrat}, we characterize all  {\it universal\/} Poisson brackets on the
space of edge weights   of a given network that respect the natural operation of concatenation of
networks, see Proposition~\ref{6param}. Furthermore, we establish that the family
of universal brackets induces a linear two-parameter family of Poisson brackets on
the space of face and trail weights, see Theorem~\ref{PSviay}, and hence on 
boundary measurement matrices, see Theorem~\ref{PSR}. This family
depends on a mutual location of sources and sinks, but not on the network itself.
We  provide an explicit description of this family in Propositions~\ref{PSRE} and~\ref{PSRE2}. An important tool in the proof of Theorem~\ref{PSR} is the realization theorem, see Theorem~\ref{realiz}, that claims that any rational matrix function can be realized as the boundary measurement matrix of a network with a given set of sources and sinks. 
Finally, if the sources and the sinks are separated, that is, all sources belong to one of the bounding circles of the annulus, and all sinks to the other bounding circle, one of the
generators of the 2-parametric family can be identified with the R-matrix Sklyanin bracket corresponding to the trigonometric R-matrix, see Theorem~\ref{sklyaninbr}.

In Section~\ref{PSongrloop}, the  boundary measurement map  defined by a network with $k$ sources, $n-k$ sinks and $n_1\le n$ boundary vertices on the outer boundary circle is extended to the Grassmannian  boundary measurement map into the
space $LG_k(n)$ of Grassmannian loops.  
The Poisson family on boundary measurement matrices allows us to
equip  $LG_k(n)$ with a  two-parameter family of Poisson brackets
$\P_{\alpha,\beta}^{n_1}$ in such a way that for any choice of a universal Poisson 
bracket on edge weights there is a unique member of $\P_{\alpha,\beta}^{n_1}$ 
that makes the Grassmannian boundary measurement map Poisson, see Theorem~\ref{PSLGr}. 
This latter family depends only on the number of sources and sinks and on the distribution of the boundary vertices between the bounding circles of the annulus. 
The main tool in the proof of Theorem~\ref{PSLGr} is the path reversal operation on networks and its properties, see 
Theorem~\ref{pathrev}. 

\section{Perfect planar networks and boundary measurements}
\label{netsBM2}

\subsection{Networks, cuts, paths and weights}
\label{PandW2}

Let $G=(V,E)$ be a directed planar graph drawn inside an annulus 
with the vertex set $V$ and the edge set $E$. 
%The graph is allowed to have loops and multiple edges.
Exactly $n$ of its vertices are located on the boundary circles of
the annulus and are called {\it boundary vertices\/}; $n_1\ge0$ of them lie on the outer circle, and $n_2=n-n_1\ge0$ on the inner circle. The graph is considered up to an isotopy relative to the boundary (with fixed boundary vertices).

Each boundary vertex is marked as a source or a sink. A {\it source\/} is
a vertex with exactly one outcoming edge and no incoming edges.
{\it Sinks\/} are defined in the same way, with the direction of the single edge
reversed.
All the internal vertices of $G$ have degree~$3$ and are of two types: either they have exactly one
incoming edge, or exactly one outcoming edge. The vertices of the first type are called
(and shown in figures) {\it white}, those of the second type, {\it black}.

A {\it cut\/} $\rho$ is an oriented non-selfintersecting curve starting at a {\it base point\/} on the inner circle and ending at a base point on the outer circle  considered up to an isotopy relative to the boundary (with fixed endpoints). 
%Without loss of generality 
We assume that the base points of the cut are distinct from the boundary vertices of $G$. 
%and that none of the internal vertices of $G$ lie on the cut. 
%Therefore, for any edge $e\in E$ its {\it intersection index\/} with $\rho$ is %well defined as an isotopy invariant. We denote this index by $\ind(e)$.
%Moreover, we assume that each intersection of an edge of $G$ with the cut is %transversal.
For an arbitrary oriented curve $\gamma$ with endpoints not lying on the cut $\rho$ we denote by $\ind(\gamma)$ the algebraic intersection number of $\gamma$ and $\rho$. Recall that each transversal intersection point of $\gamma$ and $\rho$ contributes to this number~$1$ if the oriented tangents to $\gamma$ and $\rho$ at this point form a positively oriented basis, and $-1$ otherwise. Non-transversal intersection points are treated in a similar way.

Let $x_1,\dots,x_d$ be independent variables.
A {\it perfect planar network in an annulus\/} $N=(G,\rho,w)$ is obtained from a graph $G$ equipped with a cut $\rho$ as above by assigning a weight $w_e\in \Z(x_1,\dots,x_d)$ to each edge $e\in E$. Below we occasionally write ``network'' instead of ``perfect planar network in an annulus''. Each network defines a rational map $w: \R^d\to \R^{|E|}$;
the {\it space of edge weights\/} $\EE_N$ is defined as the intersection of the image of $w$ with $(\R\setminus 0)^{|E|}$. In other words, a point in $\EE_N$ is a graph $G$ as above with edges weighted by nonzero real numbers obtained by specializing the variables $x_1,\dots,x_d$ in the expressions for $w_e$.

 A {\it path\/} $P$ in $N$ is an alternating sequence $(v_1, e_1,v_2,\dots,e_r, v_{r+1})$ of
vertices and edges such that $e_i=(v_i,v_{i+1})$ for any $i\in [1,r]$.
Sometimes we omit the names of the vertices and write $P=(e_1,\dots,e_r)$. 
A path is called a {\it cycle\/} if $v_{r+1}=v_1$ and a {\it simple cycle\/} if additionally $v_i\ne v_j$ for any other pair $i\ne j$.

To define the weight of a path we need the following construction.
Consider a closed oriented polygonal plane curve $C$. Let $e'$ and $e''$ be two consequent oriented segments of $C$, and let $v$ be their common vertex. We assume for simplicity that for any such pair $(e',e'')$, the cone spanned by $e'$ and $e''$  is not a line; in other words, if $e'$ and $e''$ are collinear, then they have the same direction. 
Observe that since $C$ is not necessary simple, there might be other edges of $C$ incident to $v$. 
%(see Figure~\ref{fig:concord} below). 
Let $l$ be an arbitrary oriented line. Define $c_l(e',e'')\in \Z/2\Z$ in the following way:
$c_l(e',e'')=1$ if the directing vector of $l$ belongs to the interior of the cone spanned by $e'$ and $e''$ 
and $c_l(e',e'')=0$  otherwise.
%(see Figure~\ref{fig:concord} for examples). 
Define $c_l(C)$ as the sum of $c_l(e',e'')$ over all pairs of consequent segments in $C$. 
It follows immediately from Theorem~1 in \cite{GrSh} that $c_l(C)$ does not depend on $l$, provided $l$ is not collinear to any of the segments in $C$.
The common value of $c_l(C)$ for different choices of $l$ is denoted by $c(C)$ and called the {\it concordance number\/} of $C$. In fact, $c(C)$ equals $\mod 2$ the rotation number of $C$; the definition of the latter is similar, but more complicated.

In what follows we assume without loss of generality that $N$ is drawn in such a way that all its edges and the cut are smooth curves. Moreover, any simple path in $N$ is a piecewise-smooth curve with no cusps,
%at any internal vertex of $N$, all the three edges intersect transversally, 
%no two tangent lines to the edges are collinear, 
at any boundary vertex of $N$ the edge and the circle intersect transversally,
%the tangent line to the edge is not tangent to the circle, 
and the same holds for the cut at both of its  base points.  
%moreover, the outer and the inner circle will be also drawn as polygonal curves.  
Given a path $P$ between a source $b'$ and a sink $b''$, we define a closed %polygonal 
piecewise-smooth curve $C_P$ in the following way: if both $b'$ and $b''$ belong to the same circle, $C_P$ is obtained by adding to $P$ the path between $b''$ and $b'$ that goes counterclockwise along the boundary of the corresponding circle. Otherwise, if $b'$ and $b''$ belong to distinct circles, $C_P$ is obtained by adding to $P$ the path that starts at $b''$, goes counterclockwise along the corresponding circle to the base point of the cut, follows the cut to the other base point and then goes counterclockwise along the other circle up to $b'$. 
Clearly, the obtained curve $C_P$ does not have cusps, so its concordance number $c(C_P)$ can be defined in a straightforward manner via polygonal approximation.

Finally the weight of $P$ is defined as
\begin{equation}\label{weightviapath}
w_P=w_P(\lambda)=(-1)^{c(C_P)-1}\lambda^{\ind(P)}\prod_{e\in P} w_e,
\end{equation}
where $\lambda$ is an auxiliary independent variable.
%called the {\it spectral parameter} 
%and $\ind(P)$ is the intersection number of $P$ and $\rho$. 
Occasionally, it will be convenient to assume that the internal vertices of $G$ do not lie on the cut and to rewrite the above formula as
\begin{equation}\label{weightviaedge}
w_P=(-1)^{c(C_P)-1}\prod_{e\in P} \bar w_e,
\end{equation}
where $\bar w_e=w_e\lambda^{\ind(e)}$ are {\it modified edge weights\/}. Observe that the weight of a path is a relative isotopy invariant, while modified edge weights are not. 
The weight of an arbitrary cycle in $N$ is defined in the same way via the concordance number of the cycle. 

If edges $e_i$ and $e_j$ in $P$ coincide and $i<j$, the path $P$ can be decomposed into the
path $P'=(e_1,\dots,e_{i-1},e_{i}=e_j,e_{j+1},\dots,e_{r})$
and the cycle $C^0=(e_{i},e_{i+1},\dots,e_{j-1})$. Clearly, $c(C_P)=c(C_{P'})+c(C^0)$, and hence 
\begin{equation}\label{offcycle}
w_P=-w_{P'}w_{C^0}.
\end{equation}

An example of a perfect planar network in an annulus is shown
in Fig.~\ref{fig:pergrann} on the left. It has two sources, $b$ on the outer circle 
and $b''$ on the inner circle, and one sink $b'$ on the inner circle. Each edge $e_i$ is labeled by its
weight. The cut is shown by the dashed line. 
The same network is shown in Fig.~\ref{fig:pergrann} on the right; 
it differs from the original picture by an isotopic deformation of the cut.

\begin{figure}[ht]
\begin{center}
\includegraphics[height=5cm]{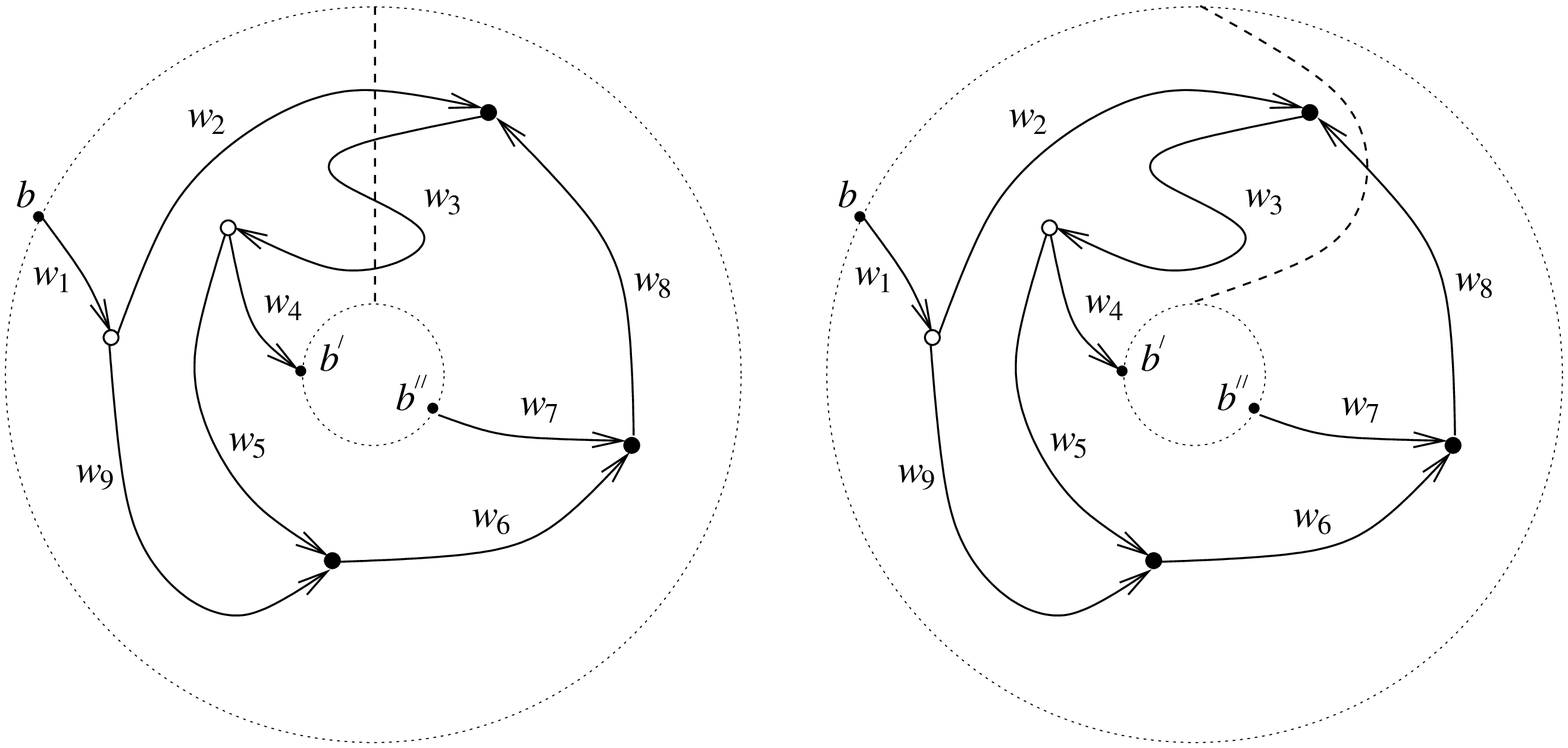}
\caption{A perfect planar network in an annulus}
\label{fig:pergrann}
\end{center}
\end{figure}

Consider the path $P_1=(e_1,e_2,e_3,e_4)$ from $b$ to $b'$. Its algebraic intersection number with the cut equals~$0$ (in the left picture, two intersection points contribute~$1$ each, and two other intersection points contribute~$-1$ each; in the right picture there are no intersection points). The concordance number of the corresponding closed curve $C_{P_1}$ equals~$1$. Therefore,~(\ref{weightviapath}) gives $w_{P_1}=w_1w_2w_3w_4$. On the other hand, the modified weights for the left picture are given by $\bar w_1=w_1$, $\bar w_2=\lambda w_2$, $\bar w_3=\lambda^{-1}w_3$, $\bar w_4=w_4$, and hence, computation via~(\ref{weightviaedge}) gives $w_{P_1}=\bar w_1\bar w_2\bar w_3\bar w_4=w_1w_2w_3w_4$. Finally, the modified weights for the relevant edges in the right picture coincide with the original weights, and we again get the same result.

Consider the path $P_2=(e_7,e_8,e_3,e_4)$ from $b''$ to $b'$. Its algebraic intersection number with the cut equals~$-1$, and the concordance number of the corresponding closed curve $C_{P_2}$ equals~$1$. Therefore, $w_{P_2}=\lambda^{-1}w_3w_4w_7w_8$. The same result can be obtained by using modified weights.

Finally, consider the path $P_3=(e_1,e_2,e_3,e_5,e_6,e_8,e_3,e_4)$ from $b$ to $b'$. Clearly,
$w_{P_3}=-\lambda^{-1}w_1w_2w_3^2w_4w_5w_6w_8$. The path $P_3$ can be decomposed into the path $P_1$ as above and a cycle $C^0=(e_3,e_5,e_6,e_8)$ with the weights $w_{P_1}=w_1w_2w_3w_4$ and $w_{C^0}=\lambda^{-1}w_3w_5w_6w_8$, hence relation~(\ref{offcycle}) yields the same expression for $w_{P_3}$ as before.

%Let us study the effect of moving a base point of the cut on the weights of paths. 
Let us see how moving a base point of the cut affects the weights of paths.
Let $N=(G,\rho,w)$ and $N'=(G,\rho',w)$ be two networks with the same graph and the same weights, and assume that the cuts $\rho$ and $\rho'$ are not isotopic. 
More exactly, let us start moving a base point of the cut in the counterclockwise direction. 
Assume that $b$ is the first boundary vertex in the counterclockwise direction from the base point of $\rho$ that is being moved. Clearly, nothing is changed while the base point and $b$ do not interchange.
Let $\rho'$ be the cut obtained after the interchange, and assume that no other interchanges occured.
Then the relation between the weight $w_P$ of a path $P$ in $N$ and its weight $w'_P$ in $N'$ is given by the following proposition.

\begin{proposition}\label{movecut}
For $N$ and $N'$ as above,
\[
w'_P(\lambda)=((-1)^{\alpha(P)}\lambda)^{\beta(b,P)}w_P,
\]
where $\alpha(P)$ equals~$0$ if the endpoints of $P$ lie on the same circle and~$1$ otherwise, and
\[
\beta(b,P)=\begin{cases}
1 \qquad\text{if $b$ is the sink of $P$},\\
-1\qquad\text{if $b$ is the source of $P$},\\ 
0 \qquad\text{otherwise}.
\end{cases}
\]
\end{proposition}

\proof 
The proof is straightforward.

\endproof

\begin{figure}[ht]
\begin{center}
\includegraphics[height=5cm]{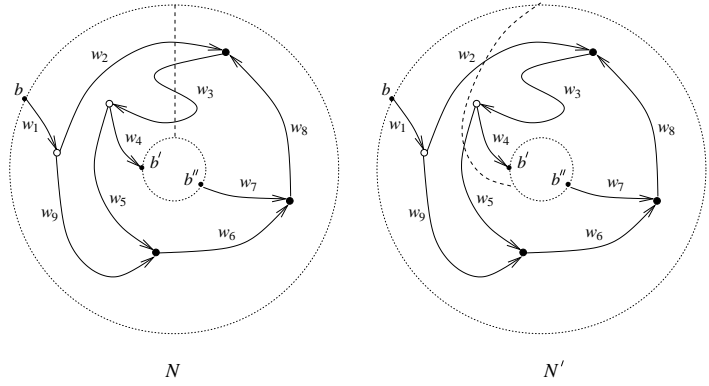}
\caption{Moving the base point of the cut}
\label{fig:pergrann2}
\end{center}
\end{figure}

For example, consider the networks $N$ and $N'$ shown in Fig.~\ref{fig:pergrann2}. 
The base point of the cut lying on the inner circle interchanges positions with the sink $b'$. The path $P_1$ from the previous example goes from the outer circle to the inner circle, so $\alpha(P_1)=1$; besides, $b'$ is the sink of $P_1$, so $\beta (b',P_1)=1$. Therefore, by Proposition~\ref{movecut}, $w'_{P_1}=-\lambda w_{P_1}=-\lambda   w_1w_2w_3w_4$, which coincides with the value obtained via~(\ref{weightviaedge}). The path $P_2$ from the previous example starts and ends at the inner circle, so $\alpha(P_2)=0$. Besides, $\beta(b',P_2)=1$, and hence, by Proposition~\ref{movecut}, $w'_{P_2}=\lambda w_{P_2}= w_1w_2w_3w_4$; once again, this coincides with the value obtained via~(\ref{weightviaedge}).

\subsection{Boundary measurements}\label{BM2} Given a
perfect planar network in an annulus as above, we label its boundary vertices $b_1,\dots,b_n$ in the following way. The boundary vertices lying on the outer circle are labeled $b_1,\dots,b_{n_1}$ in the counterclockwise order starting from the first vertex that follows after the base point of the cut. The boundary vertices lying on the inner circle are labeled $b_{n_1+1},\dots,b_n$ in the clockwise order starting from the first vertex that follows after the base point of the cut. For example, for the network $N$ in Fig.~\ref{fig:pergrann2}, the boundary vertices are labeled as $b_1=b$, $b_2=b''$, $b_3=b'$, while for the network $N'$,
one has $b_1=b$, $b_2=b'$, $b_3=b''$. 

The number of sources lying on the outer circle is denoted by $k_1$, and the corresponding set of indices, by
$I_1\subset [1,n_1]$; the set of the remaining $m_1=n_1-k_1$ indices is denoted by $J_1$. Similarly, the number of sources lying on the inner circle is denoted by $k_2$, and the corresponding set of indices, by
$I_2\subset [n_1+1,n]$; the set of the remaining $m_2=n_2-k_2$ indices is denoted by $J_2$. Finally, we denote $I=I_1\cup I_2$ and $J=J_1\cup J_2$; the cardinalities of $I$ and $J$ are denoted $k=k_1+k_2$ and $m=m_1+m_2$.

Given a source $b_i$, $i\in I$, and a sink $b_j$, $j\in J$, we define the {\it boundary measurement\/} $M(i,j)$ as the
sum of the weights of all paths starting at $b_i$ and ending at $b_j$. Assume first that the weights of the paths are calculated via~(\ref{weightviaedge}).
The boundary measurement thus defined is a formal infinite series
in variables $\bar w_e$, $e\in E$. The following 
%analog of Proposition~\ref{sfree} 
proposition holds true in $\Z[[\bar w_e, e\in E]]$.

\begin{proposition}\label{sfree2}
Let $N$ be a perfect planar network in an annulus, then each boundary measurement in $N$ is a rational function in the modified weights $\bar w_e$, $e\in E$. 
\end{proposition}

\proof
The proof by induction on the number of internal vertices literally follows the proof of the similar statement in \cite{GSV3}. 
The only changes are that modified weights are used instead of original weights and that the counterclockwise cyclic order $\prec$ is replaced by the cyclic order $\mod n$ induced by the labeling.
\endproof

Taking into account the definition of the modified weights, we immediately get the following corollary.

\begin{corollary}\label{sfree3}
Let $N$ be a perfect planar network in an annulus, then each boundary measurement in $N$ is a rational function in the parameter $\lambda$ and the weights $w_e$, $e\in E$. 
\end{corollary}

For example, the boundary measurement $M(1,2)$ in the network $N'$ shown on Fig.~\ref{fig:pergrann2} equals
\[
\frac{w_1w_3w_4(w_6w_8w_9-\lambda w_2)}{1+\lambda^{-1}w_3w_5w_6w_8}.
\]

Boundary measurements can be organized into a $k\times m$ {\it
boundary measurement matrix\/} $M_N$ exactly as in the case of planar networks in the disk: let $I=\{i_1<i_2<\dots<i_k\}$ and $J=\{j_1<j_2<\dots<j_m\}$, then 
%We define 
$M_N=(M_{pq})$, $p\in [1,k]$, $q\in [1,m]$, where $M_{pq}=M({i_p},{j_q})$. 
Let $\Rat_{k,m}$ stand for the space of real rational $k\times m$ matrix functions
in one variable. Then each network $N$ defines a map
$\EE_N\to \Rat_{k,m}$ given by $M_N$ and called the {\it boundary
measurement map\/} corresponding to $N$. 

The boundary measurement matrix has a block structure
\[
M_N=\begin{pmatrix} 
M_1 & M_2\\
M_3 & M_4
\end{pmatrix},
\]
where $M_1$ is $k_1\times m_1$ and $M_4$ is $k_2\times m_2$. Moving a base point of the cut changes the weights of paths as described in Proposition~\ref{movecut}, which affects $M_N$ in the following way. Define
\[
\Lambda_+=\begin{pmatrix}
0 & 1 & \dots & 0\\
\vdots & \vdots & \ddots & \vdots\\
0 & 0 & \dots & 1\\
\lambda^{-1} & 0 & \dots & 0
\end{pmatrix}, \qquad
\Lambda_-=\begin{pmatrix}
0 & 1 & \dots & 0\\
\vdots & \vdots & \ddots & \vdots\\
0 & 0 & \dots & 1\\
-\lambda^{-1} & 0 & \dots & 0
\end{pmatrix},
\]
then interchanging the base point of the cut with $b_1$ implies transformation
\[
\begin{pmatrix} 
M_1 & M_2\\
M_3 & M_4
\end{pmatrix}
\mapsto
\begin{pmatrix} 
\Lambda_+ M_1 & \Lambda_- M_2\\
M_3 & M_4
\end{pmatrix},
\]
if $b_1$ is a source and
\[
\begin{pmatrix} 
M_1 & M_2\\
M_3 & M_4
\end{pmatrix}
\mapsto
\begin{pmatrix} 
M_1\Lambda_+^{-1}  & M_2\\
M_3\Lambda_- ^{-1} & M_4
\end{pmatrix},
\]
if $b_1$ is a sink, 
while interchanging the base point of the cut with $b_n$ yields
\[
\begin{pmatrix} 
M_1 & M_2\\
M_3 & M_4
\end{pmatrix}
\mapsto
\begin{pmatrix} 
M_1 & M_2\\
\Lambda_+^T M_3 & \Lambda_-^T M_4
\end{pmatrix}.
\]
if $b_n$ is a source and
\[
\begin{pmatrix} 
M_1 & M_2\\
M_3 & M_4
\end{pmatrix}
\mapsto
\begin{pmatrix} 
M_1  & M_2(\Lambda_+^{-1})^T\\
M_3 & M_4(\Lambda_- ^{-1})^T
\end{pmatrix},
\]
if $b_1$ is a sink. Note that $\Lambda_+$ and $\Lambda_-$ are $k_1\times k_1$ in the first case, $m_1\times m_1$ in the second case, $k_2\times k_2$ in the third 
case and $m_2\times m_2$ in the fourth case.

\subsection{The space of face and trail weights}\label{cohomol}
Let $N=(G,\rho,w)$ be a perfect network. 
Consider the $\Z$-module $\Z^E$ generated by the edges of $G$. Clearly, points of $\EE_N$ can be identified with the elements of $\Hom(\Z^E,\R^*)$ via $w(\sum n_ie_i)=\prod w_{e_i}^{n_i}$, where $\R^*$ is the abelian multiplicative group $\R\setminus 0$. Further, consider the $\Z$-module $\Z^V$ generated by the vertices of $G$ and its $\Z$-submodule $\Z^{V_0}$ generated by the internal vertices. An arbitrary element $\varphi\in \Hom(\Z^V,\R^*)$ acts on $\EE_N$ as follows: if $e=(u,v)$ then 
\[
w_e\mapsto w_e\frac {\varphi(v)}{\varphi(u)}.
\]
Therefore, the weight of a path between the boundary vertices $b_i$ and $b_j$ is multiplied by $\varphi(b_j)/\varphi(b_i)$. It follows that the {\it gauge group\/} $\GG$, which preserves the weights of all paths between  boundary vertices, consists of all $\varphi\in \Hom(\Z^V,\R^*)$ such that $\varphi(b)=1$ for any boundary vertex $b$, and can be identified with $\Hom(\Z^{V_0},\R^*)$. Thus, the boundary measurement map $\EE_N\to\Rat_{k,m}$ factors through the quotient space $\FF_N=\EE_N/\GG$ as follows:
$M_N=M_N^\FF\circ y$, where $y: \EE_N\to \FF_N$ is the projection and $M_N^\FF$ is a map
$\FF_N\to \Rat_{k,m}$. The space $\FF_N$ is called the {\it space of face and trail weights\/} for the following reasons. 

First, by considering the cochain complex
\[
0\to \GG\to \EE_N\to 0
\]
with the coboundary operator $\delta:\GG\to \EE_N$ defined by $\delta(\varphi)(e)=\varphi(v)/\varphi(u)$ for $e=(u,v)$, we can identify $\FF_N$ with the first relative cohomology $H^1(G,\partial G;\R^*)$ of the complex, where $\partial G$ is the set of all boundary vertices of $G$.

Second, consider a slightly more general situation, when the annulus is replaced by an arbitrary Riemann surface $\Sigma$ with the boundary $\partial\Sigma$, and $G$ is embedded into $\Sigma$ in such a way that all vertices of degree 1 belong to $\partial \Sigma$ (boundary vertices). Then the exact sequence of relative cohomology with coefficients in $\R^*$ gives
\begin{align*}
0\to H^0(G\cup\partial\Sigma,\partial\Sigma)\to &H^0(G\cup\partial\Sigma)\to
H^0(\partial\Sigma)\to\\ &H^1(G\cup\partial\Sigma,\partial\Sigma)\to H^1(G\cup\partial\Sigma)\to H^1(\partial\Sigma)\to 0.
\end{align*}

Evidently, $H^1(G,\partial G)\simeq H^1(G\cup\partial\Sigma,\partial\Sigma)$. Next,
$H^0(G\cup\partial\Sigma,\partial\Sigma)=0$, since each connected component of $G$ is connected to at least one connected component of $\partial\Sigma$, and hence
\[
H^1(G,\partial G)\simeq H^1(G\cup\partial\Sigma)/H^1(\partial\Sigma)\oplus H^0(\partial\Sigma)/H^0(G\cup\partial\Sigma)
%\ker (H^1(G\cup\partial\Sigma)\to H^1(\partial\Sigma))\oplus
%H^0(\partial\Sigma)/\im(H^0(G\cup\partial\Sigma))
=\FF_N^f\oplus \FF_N^t.
\]

The space $\FF_N^f$ can be described as follows. Graph $G$ divides $\Sigma$ into a finite number of connected components called
\emph{faces}. The boundary of each face consists of edges of $G$ and, possibly, of several arcs of $\partial\Sigma$. A face is called {\it bounded\/} if its boundary contains only edges of $G$ and {\it unbounded\/} otherwise. 
%In this Section we additionally require that each edge of $G$ belongs to a path from a %source to a sink. This is a technical condition that ensures that the two faces separated %by an edge are distinct. Clearly, the edges that violate this condition do not influence %the boundary measurement map and may be eliminated from the graph.

Given a face $f$, we define its {\it face weight\/} $y_f$ as the function on $\EE_N$ that assigns to the edge weights $w_e$, $e\in E$, the value
\begin{equation}\label{faceweight}
y_f=\prod_{e\in\partial f}w_e^{\gamma_e},
\end{equation}
where $\gamma_e=1$ if the direction of $e$ is compatible with the counterclockwise orientation of the boundary $\partial f$ and $\gamma_e=-1$ otherwise. It follows immediately from the definition that face weights are invariant under the gauge group action, and hence are functions on $\FF_N^f$, and, moreover, form a basis in the space of such functions.

Consider now the space $\FF_N^t$. If $\Sigma$ is a disk, then $\FF_N^t=0$, and hence $\FF_N$ and $\FF_N^f$ coincide. This case was studied in \cite{GSV3}, 
and the space $\FF_N$ was called there the {\it space of face weights}. If $\Sigma$ is an annulus, there are two possible cases. Indeed, $\dim H^1(\partial\Sigma)=2$. The dimension of $H^1(G\cup\partial\Sigma)$ is either $2$ or $1$, depending on the existence of a trail connecting the components of $\partial\Sigma$. Here a {\it trail\/} is a sequence $(v_1,\dots,v_{k+1})$ of vertices such that either $(v_i,v_{i+1})$ or $(v_{i+1},v_i)$ is an edge in $G$ for all $i\in [1,k]$ an the endpoints $v_1$ and $v_{k+1}$ are boundary vertices of $G$. Given a trail $t$, the {\it trail weight\/} $y_t$ is defined as
\[
y_t=\prod_{i=1}^k w(v_i,v_{i+1}),
\]
where 
\[
w(v_i,v_{i+1})=\begin{cases} 
w_e\qquad\text{if $e=(v_i,v_{i+1})\in E$},\\
w_e^{-1}\qquad\text{if $e=(v_{i+1},v_i)\in E$}.
\end{cases}
\] 
Clearly, the trail weights are invariant under the action of the gauge group.

If $G$ does not contain a trail connecting the inner and the outer circles, then $\dim  H^1(G\cup\partial\Sigma)=2$, and hence $\FF_N^t=0$. Otherwise,  $\dim  H^1(G\cup\partial\Sigma)=1$,  and hence $\dim\FF_N^t=1$. The functions on $\FF_N^t$ are generated by the weight of any connecting trail.

\section{Poisson properties of the boundary
measurement map}\label{PSrat}

\subsection{Poisson structures on the spaces $\EE_N$ and $\FF_N$}\label{PSSC} 
The construction of a Poisson structure on the space $\EE_N$ for perfect planar networks in an annulus is a straightforward extension of the corresponding construction for the case of the disk studied in \cite{GSV3}. 
 Let G be a directed planar graph in an annulus as
described in Section~\ref{PandW2}.
A pair $(v,e)$ is called a {\it flag\/} if $v$ is an endpoint of $e$.
To each internal vertex $v$ of $G$ we
assign a 3-dimensional space $(\R\setminus 0)_v^3$ with coordinates $x_v^1,
x_v^2, x_v^3$. We equip each $(\R\setminus 0)_v^3$ with a Poisson bracket
$\{\cdot,\cdot\}_v$. It is convenient to assume that the 
%{\it half-edges\/} incident to $v$
flags involving $v$ 
 are labeled by the coordinates, as shown on Figure~\ref{fig:trefoils}.

\begin{figure}[ht]
\begin{center}
\includegraphics[height=1.5cm]{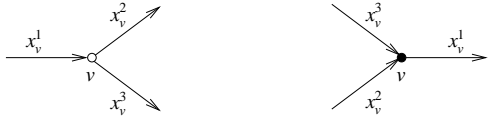}
\caption{Edge labeling for $\R_v^3$} \label{fig:trefoils}
\end{center}
\end{figure}

Besides, to each boundary vertex $b_j$ of $G$ we assign a
1-dimensional space $(\R\setminus 0)_j$ with the coordinate $x_j^1$ (in accordance with the above convention, this coordinate labels the 
%half-edge incident to 
unique flag involving
$b_j$). Define $\RR$
to be the direct sum of all the above spaces; thus, the dimension
of $\RR$ equals twice the number of edges in $G$. Note that $\RR$ is
equipped with a Poisson bracket $\{\cdot,\cdot\}_\RR$, which is
defined as the direct sum of the brackets $\{\cdot,\cdot\}_v$; that is, $\{x,y\}_\RR=0$ whenever $x$ and $y$ are not defined on the same $(\R\setminus 0)_v^3$. We
say that the bracket $\{\cdot,\cdot\}_\RR$ is {\it universal\/} if
each of $\{\cdot,\cdot\}_v$ depends only on the color of the
vertex $v$.

Define the weights $w_e$ by 
\begin{equation}\label{connect}
w_e=x^i_vx^j_u, 
\end{equation}
provided the 
% edge $e=(u,v)$ is labeled by $x^i_v$ at $v$ and by $x^j_u$ at $u$. 
flag $(v,e)$ is labeled by $x^i_v$ and the flag $(u,e)$ is labeled by $x^j_u$.
In other words, the weight of an edge is defined as the product of the weights of the two 
%half-edges constituting  
flags involving 
this edge.
Therefore, in this case the space of edge weights $\CC_N$ coincides with the entire $(\R\setminus 0)^{|E|}$, and the weights define a {\it weight map\/} $w\: (\R\setminus 0)^d\to(\R\setminus 0)^{|E|}$. 
We require the pushforward of 
$\{\cdot,\cdot\}_\RR$ to $(\R\setminus 0)^{|E|}$ by the weight map to be a well defined 
Poisson bracket; this can be regarded as an analog
of the Poisson--Lie property for groups.

\begin{proposition}\label{6param}
Universal Poisson brackets $\{\cdot,\cdot\}_\RR$ such that the
weight map $w$ is Poisson form a $6$-parametric family defined
by relations
\begin{equation}\label{6parw}
\{x^i_v,x^j_v\}_v=\alpha_{ij}x^i_vx^j_v, \quad i,j\in [1,3], i\ne
j,
\end{equation}
at each white vertex $v$ and
\begin{equation}\label{6parb}
\{x^i_v,x^j_v\}_v=\beta_{ij}x^i_vx^j_v, \quad i,j\in [1,3], i\ne
j,
\end{equation}
at each black vertex $v$.
\end{proposition}

\begin{proof} Indeed, let $v$ be a white vertex, and let $e=(v,u)$ and
$\bar e=(v,\bar u)$ be the two outcoming edges. By definition,
there exist $i,j,k,l\in [1,3]$, $i\ne j$, such that
$w_e=x_v^ix_u^k$, $w_{\bar e}=x_v^jx_{\bar u}^l$. Therefore,
$$
\{w_e,w_{\bar e}\}_N=\{x_v^ix_u^k,x_v^jx_{\bar u}^l\}_\RR=x_u^kx_{\bar
u}^l\{x_v^i,x_v^j\}_v,
$$
where $\{\cdot,\cdot\}_N$ stands for the pushforward of
$\{\cdot,\cdot\}_\RR$. Recall that the Poisson bracket in $(\R\setminus 0)_v^3$
depends only on $x_v^1$, $x_v^2$ and $x_v^3$. Hence the only
possibility for the right hand side of the above relation to be a
function of $w_e$ and $w_{\bar e}$ occurs when
$\{x^i_v,x^j_v\}_v=\alpha_{ij}x^i_vx^j_v$, as required.

Black vertices are treated in the same way.
\end{proof}

The 6-parametric family of universal Poisson brackets described in Proposition~\ref{6param} induces a 
6-parametric family of Poisson brackets $\Poi_N$ on $\EE_N$. 
Our next goal is to study the pushforward of this family to $\FF_N$ by the map $y$. 

%Let $N=(G,\rho,w)$ be a perfect planar network. In what follows it will be convenient to assume that boundary 
%vertices are colored in {\it gray}.  Define the {\it directed dual network\/} $N^*=(G^*,w^*)$ as follows. 
%Vertices of $G^*$ are the faces of $N$. Edges of $G^*$ correspond to the edges of $N$ with endpoints of 
%different colors; note that there might be several edges between the same pair of vertices in $G^*$.
%An edge $e^*$ of $G^*$ corresponding to $e$ is directed in such a way that the white endpoint of $e$ 
%(if it exists) lies to the left of $e^*$ and the black endpoint of $e$ (if it exists) lies to the right of $e$. 
%The weight $w^*(e^*)$ equals $\alpha-\beta$ if the endpoints of $e$ are white and black, $\alpha$ if the 
%endpoints of $e$ are white and gray and $-\beta$ if the endpoints of $e$ are black and gray, where $\alpha$ 
%and $\beta$ are given by
%\begin{equation}\label{cond}
%\alpha=\alpha_{23}+\alpha_{13}-\alpha_{12}, \qquad \beta=\beta_{23}+\beta_{13}-\beta_{12}.
%\end{equation}
%For a trail $t$, $t^*$ denotes the set of all edges of $N^*$ corresponding to the edges of $t$.

%An example of a directed planar network and its directed dual network is given on Fig.~\ref{fig:dualnet}.

%\begin{figure}[ht]
%\begin{center}
%\includegraphics[height=4cm]{dualnet.eps}
%\caption{Directed planar network and its directed dual} \label{fig:dualnet}
%\end{center}
%\end{figure}

\begin{theorem}\label{PSviay}
The 6-parametric family  $\Poi_N$ induces a 2-parametric family of Poisson brackets $\{\cdot,\cdot\}_{\FF_N}$ 
on $\FF_N$ with parameters $\alpha$ and $\beta$ given by
\begin{equation}\label{cond}
\alpha=\alpha_{23}+\alpha_{13}-\alpha_{12}, \qquad \beta=\beta_{23}+\beta_{13}-\beta_{12}.
\end{equation}
\end{theorem}

%given by
%\begin{align*}\label{2param}
%&\{y_f,y_{f'}\}_{\FF_N}=\left(\sum_{e^*: f\to f'} w^*(e^*)-
%\sum_{e^*: f'\to f} w^*(e^*)\right)y_fy_{f'},\\
%&\{y_f,y_{t}\}_{\FF_N}=\left(\sum_g\sum_{t^8\ni e^*: f\to g} w^*(e^*)-
%\sum_g\sum_{t^*\ni e^*: g\to f} w^*(e^*)\right)y_fy_{t}
%\end{align*}
%with $\alpha$ and $\beta$ satisfying~(\ref{cond}).

\begin{proof}
In what follows it will be convenient to assume that boundary 
vertices are colored in {\it gray}.
Let $e=(u,v)$ be a directed edge. We say that the flag $(u,e)$ is {\it positive\/}, and the flag $(v,e)$ is 
{\it negative}. The color of a flag is defined as the color of the vertex participating in the flag.

Consider first a bracket of two face weights.
Let $f$ and $f'$ be two faces of $N$. We say that a flag $(v,e)$ is {\it common\/} to $f$ and $f'$ if both $v$ 
and $e$ belong to $\partial f\cap\partial f'$. Clearly, the bracket  $\{y_f,y_{f'}\}_{\FF_N}$ can be calculated 
as the sum of the contributions of all flags common to $f$ and $f'$.

Assume that $(v,e)$ is a positive white flag common to $f$ and $f'$, see Fig.~\ref{fig:comflag}. 
Then $y_f=\dfrac{x_v^3}{x_v^2}\bar y_f$ and $y_{f'}=x_v^1x_v^2\bar y_{f'}$, where $x_v^i$ are the weights of flags 
involving $v$ and $\{x_v^i,\bar y_f\}_R=\{x_v^i,\bar y_{f'}\}_R=0$. Therefore, by~(\ref{6parw}), 
the contribution of $(v,e)$ equals $(\alpha_{12}-\alpha_{13}-\alpha_{23})y_fy_{f'}$, 
which by~(\ref{cond}) equals $-\alpha y_fy_{f'}$. 

Assume now that $(v,e)$ is a negative white flag common to $f$ and $f'$, see Fig.~\ref{fig:comflag}. In this case 
$y_f=\dfrac1{x_v^1x_v^3}\bar y_f$ and $y_{f'}=x_v^1x_v^2\bar y_{f'}$, so  the contribution of $(v,e)$ equals 
$(\alpha_{13}+\alpha_{23}-\alpha_{12})y_fy_{f'}=\alpha y_fy_{f'}$.

\begin{figure}[ht]
\begin{center}
\includegraphics[height=2.5cm]{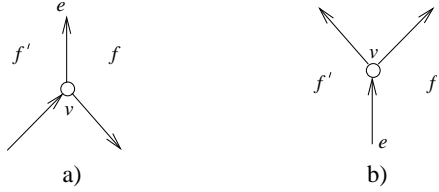}
\caption{Contribution of a white common flag: a) positive flag; b) negative flag} \label{fig:comflag}
\end{center}
\end{figure}

In a similar way one proves that the contribution of a positive black flag common to $f$ and $f'$ equals $-\beta y_fy_{f'}$, 
and the contribution of a negative black flag common to $f$ and $f'$ equals $\beta y_fy_{f'}$. 
Finally, the contributions of positive and negative gray flags are clearly equal to zero.
Therefore, the brackets $\{y_f,y_{f'}\}_{\FF_N}$ form a 2-parametric family with parameters $\alpha$ and $\beta$ 
defined by~(\ref{cond}).
 
The case of a bracket $\{y_f,y_{t}\}_{\FF_N}$ is treated in a similar way.
\end{proof}

\subsection{Induced Poisson structures on  
$\Rat_{k,m}$}\label{PSonrat}
Fix an arbitrary pair of partitions $I_1\cup J_1=[1,n_1]$, $I_1\cap J_1=\varnothing$, $I_2\cup J_2=[n_1+1,n]$, 
$I_2\cap J_2=\varnothing$, 
and denote $k=|I_1|+|I_2|$, $m=n-k=|J_1|+|J_2|$.
Let $\Net_{I_1,J_1,I_2,J_2}$ stand for the set of all perfect planar
networks in an annulus with the sources $b_i$, $i\in I_1$ and sinks $b_j$, $j\in J_1$, on the outer circle, 
sources $b_i$, $i\in I_2$ and sinks $b_j$, $j\in J_2$, on the inner circle, and edge weights $w_e$ defined by~(\ref{connect}).
We assume that the space of edge weights $\EE_N=\R^{|E|}$ is
equipped with the Poisson bracket $\{\cdot,\cdot\}_N$ obtained as the pushforward of the 6-parametric 
family $\{\cdot,\cdot\}_R$ described in
Proposition~\ref{6param}. 

\begin{theorem}\label{PSR} 
There exists a $2$-parametric family of Poisson brackets $\Poi_{I_1,J_1,I_2,J_2}$ on
$\Rat_{k,m}$ with the following property: for any choice of
parameters $\alpha_{ij}$, $\beta_{ij}$
in~(\ref{6parw}), (\ref{6parb}) this family contains a unique
Poisson bracket on $\Rat_{k,m}$ such that for any network
$N\in\Net_{I_1,J_1,I_2,J_2}$ the map $M_N\:(\R\setminus 0)^{|E|}\to \Rat_{k,m}$ is Poisson.
\end{theorem}

\proof
First of all, we use the factorization $M_N=M_N^{\FF}\circ y$ to decrease the number of parameters. 
By Theorem~\ref{PSviay}, it suffices to consider the 2-parametric
family
\begin{equation}\label{2parw}
\{\bar x^2_v,\bar x^3_v\}_v=\alpha \bar x^2_v\bar x^3_v, \quad
\{\bar x^1_v,\bar x^2_v\}_v=\{\bar x^1_v,\bar x^3_v\}_v=0
\end{equation}
and
\begin{equation}\label{2parb}
\{\bar x^2_v,\bar x^3_v\}_v=\beta \bar x^2_v\bar x^3_v, \quad
\{\bar x^1_v,\bar x^2_v\}_v=\{\bar x^1_v,\bar x^3_v\}_v=0
\end{equation}
with $\alpha$ and $\beta$ defined by~(\ref{cond}), 
instead of the 6-parametric
family~(\ref{6parw}),~(\ref{6parb}) 

The rest of the proof consists of two major steps. First, we compute the induced Poisson bracket on the image of the boundary measurement map. More exactly, we show that the bracket $\Poi_N$ of any pair of pullbacks of coordinate functions on the image can be expressed in terms of pullbacks of other coordinate functions, and that for fixed $I_1,J_1,I_2,J_2$ these expressions do not depend on the network $N\in\Net_{I_1,J_1,I_2,J_2}$.
Second, we prove that any rational matrix function belongs to the image of the boundary measurement map (for a sufficiently large $N\in\Net_{I_1,J_1,I_2,J_2}$), and therefore $\Poi_N$ induces $\Poi_{I_1,J_1,I_2,J_2}$ on $\Rat_{k,m}$. This approach allows us to circumvent technical difficulties one encounters when attempting to check the Jacobi identity in the image in a straightforward way.

To compute the induced Poisson bracket on the image of the boundary measurement map, we consider coordinate functions $\val_t\:f\mapsto f(t)$ that assign to any $f\in \Rat_{1,1}$ its value at point $t$. Given a pair of two matrix entries, it suffices to calculate the bracket between arbitrary pair of functions $\val_t$ and $\val_s$ defined on two copies of  $\Rat_{1,1}$ representing these entries.
Since the pullback of $\val_t$ is the corresponding component of $M_N(t)$, we have to deal with expressions of the form $\{M_{pq}(t),M_{\bar p\bar q}(s)\}_N$.

%In what follows we write $f(t)$ instead of $\val_t(f)$.

%one has to pick coordinate functions on the space $\Rat_{1,1}$ of rational %functions and, for any choice of two matrix entries, to calculate the bracket %%between 
%for any pair of coordinate functions defined on two copies of  $\Rat_{1,1}$ %representing the chosen matrix entries. A convenient choice 

%For example, if one picks the coefficients of the numerator and the denominator as %coordinate functions on $\Rat_{1,1}$, and the chosen entries are $(p,q)$ and %$(\bar p,\bar q)$, then one has to calculate the bracket 
%%between 
%for any pair of coefficients of $M_{pq}$ and $M_{\bar p\bar q}$. 

%A convenient way to organize such calculations independently of the choice of %coordinate functions is to introduce %two auxiliary variables $t$ and $s$ 
%an additional auxiliary variable $t$ (interpreted as the variable in the second %copy of $\Rat_{1,1}$) and to calculate the bracket 
%%between 
%of rational functions $M_{pq}(\lambda)$ and $M_{\bar p\bar q}(t)$. The brackets 
%%between 
%for the coordinate functions on the corresponding two copies of $\Rat_{1,1}$ can %be restored from these expressions. 

To avoid overcomplicated formulas we  consider separately two particular representatives of the family~(\ref{2parw}),~(\ref{2parb}): 1) $\alpha=-\beta=1$, and 2) $\alpha=\beta=1$. Any member of the family can be represented as a linear combination of the above two.

Denote by $\Poi^1_N$ the member of the 2-parametric
family~(\ref{2parw}),~(\ref{2parb}) corresponding to the case 
%on the image of the boundary measurement map 
$\alpha=-\beta=1$. Besides, define $\sigma_=(i,j,i',j')=\s(i'-i)-\s(j'-j)$; clearly, $\sigma_=(i,j,i',j')$ is closely related to $s_=(i,j,i',j')$ defined and studied in~\cite{GSV3}. 
The bracket induced by $\Poi^1_N$ on the image of the boundary measurement map is completely described by the following statement.

\begin{proposition}\label{PSRE}
{\rm (i)} Let $i_p, i_{\bar p} \in [1,n_1]$ and $1\leq \max\{i_p,i_{\bar p}\}<j_{\bar q}<j_q\leq n$, then
\begin{equation}\label{psre1}
\{M_{pq}(t),M_{\bar p\bar q}(s)\}^1_N=
\sigma_=(i_p,j_q,i_{\bar p},j_{\bar q})M_{p\bar q}(s)M_{\bar p q}(t)
-\frac2{t-s}\Phi_{pq}^{\bar p \bar q}(t,s),
\end{equation}
where 
\[
\Phi_{pq}^{\bar p \bar q}(t,s)=\left\{\begin{array}{ll}
\big(M_{p\bar q}(t)-M_{p\bar q}(s)\big)
\big(s M_{\bar p q}(t)-t M_{\bar p q}(s)\big),
&\qquad j_{\bar q}<j_q\leq n_1,\\
s M_{\bar p q}(t)\big(M_{p\bar q}(t)-M_{p\bar q}(s)\big), 
&\qquad j_{\bar q}\leq n_1< j_q,\\
s\big(M_{p\bar q}(t)M_{\bar p q}(s)-M_{p\bar q}(s)M_{\bar p q}(t)\big),
&\qquad n_1< j_{\bar q}<j_q.
\end{array}\right.
\]

{\rm (ii)} Let $j_{\bar q},j_q \in [n_1+1,n]$ and 
$1\leq i_p<i_{\bar p}<\min\{j_{\bar q},j_q\}\leq n$, then
\begin{equation}\label{psre2}
\{M_{pq}(t),M_{\bar p\bar q}(s)\}^1_N=
\sigma_=(i_p,j_q,i_{\bar p},j_{\bar q})M_{p\bar q}(t)M_{\bar p q}(s)
-\frac2{t-s}\Psi_{pq}^{\bar p \bar q}(t,s),
\end{equation}
where 
\[
\Psi_{pq}^{\bar p \bar q}(t,s)=\left\{\begin{array}{ll}
t\big(M_{p\bar q}(t)M_{\bar p q}(s)-M_{p\bar q}(s)M_{\bar p q}(t)\big),
&\qquad i_p<i_{\bar p}\leq n_1,\\
-t M_{p \bar q}(t)\big(M_{\bar p q}(t)-M_{\bar p q}(s)\big),
&\qquad i_p\leq n_1<i_{\bar p},\\
-\big(t M_{p\bar q}(t)-s M_{p\bar q}(s)\big)
\big(M_{\bar p q}(t)-M_{\bar p q}(s)\big),
&\qquad n_1<i_p<i_{\bar p}.
\end{array}\right.
\]

{\rm (iii)} Let $1\leq i_p=i_{\bar p}<j_q=j_{\bar q}\leq n$, then
\begin{align}\label{psre3}
\{M_{pq}(t),&M_{p q}(s)\}^1_N\\
&=\left\{\begin{array}{ll}
-\frac{2}{t-s} 
\big(M_{pq}(t)-M_{pq}(s)\big)
\big(s M_{p q}(t)-t M_{p q}(s)\big),
&\qquad i_p<j_q\leq n_1,\\
0,
&\qquad i_p\leq n_1<j_q.
\end{array}\right.\nonumber
\end{align}

{\rm (iv)} Let $1\leq i_p<\min\{i_{\bar p},j_q,j_{\bar q}\}$, then
\begin{align}\label{psre4}
\{M_{pq}(t),&M_{\bar p\bar q}(s)\}^1_N\\
&=\left\{\begin{array}{ll}
\frac{2t}{t-s}
\big(M_{p\bar q}(t)-M_{p\bar q}(s)\big)
\big(M_{\bar p q}(t)-M_{\bar p q}(s)\big),
&\qquad j_{\bar q}<i_{\bar p}<j_q\leq n_1,\\
-\frac{2t}{t-s}
\big(M_{p\bar q}(t)M_{\bar p q}(t)-M_{p\bar q}(s)M_{\bar p q}(s)\big),
&\qquad j_q\leq n_1<i_{\bar p}<j_{\bar q},\\
0,
&\qquad j_{\bar q}\leq n_1<i_{\bar p}<j_q.
\end{array}\right.\nonumber
\end{align}
\end{proposition}

\proof
Let us first make sure that relations~(\ref{psre1})--(\ref{psre4}) indeed allow to compute $\{M_{pq}(t),M_{\bar p\bar q}(s)\}^1_N$ for any $p,\bar p\in [1,k]$ and $q,\bar q\in [1,m]$. This is done by employing the following three techniques:

-- moving a base point of the cut;

-- reversing the direction of the cut;

-- reversing the orientation of boundary circles.

The first of the above techniques has been described in detail in Section~\ref{PandW2}. For example, let $1\leq j_q<i_p<i_{\bar p}<j_{\bar q}\leq n_1$. This case is not covered explicitly by relations~(\ref{psre1})--(\ref{psre4}). Consider the network $N'$ obtained from $N$ by moving the base point of the cut on the outer circle counterclockwise and interchanging it with $j_q$. In this new network one has $1\leq i_{p'}<i_{\bar p'}<j_{\bar q'}<j_q'= n_1$, so the conditions of Proposition~\ref{PSRE}(i) are satisfied and~(\ref{psre1}) yields
\begin{align*}
\{M_{p'q'}(t),M_{\bar p'\bar q'}(s)\}^1_{N'}=
2&M_{p'\bar q'}(s)M_{\bar p' q'}(t)\\
&-\frac2{t-s}\big(M_{p'\bar q'}(t)-M_{p'\bar q'}(s)\big)
\big(s M_{\bar p' q'}(t)-t M_{\bar p' q'}(s)\big).
\end{align*}
By Lemma~\ref{movecut},
\[
M_{p'q'}(t)=t M_{pq}(t),\quad
M_{\bar p'\bar q'}(t)=M_{\bar p\bar q}(t),\quad
M_{p'\bar q'}(t)=M_{p\bar q}(t),\quad
M_{\bar p' q'}(t)=t M_{\bar p' q'}(t).
\]
Finally, $\Poi_N=\Poi_{N'}$ for any pair of edge weights, so we get
\[
\{M_{pq}(t),M_{\bar p\bar q}(s)\}^1_{N}=
2M_{p\bar q}(s)M_{\bar p q}(t)
-\frac{2s}{t-s}\big(M_{p\bar q}(t)-M_{p\bar q}(s)\big)
\big(M_{\bar p q}(t)-M_{\bar p q}(s)\big).
\]

Reversing the direction of the cut transforms the initial network $N$ to a new network $N'$;  the graph $G$ remains the same, while the labeling of its boundary vertices is changed. Namely, the $n_1'=n_2$ boundary vertices lying on the inner circle are labeled $b_1,\dots,b_{n'_1}$ in the clockwise order starting from the first vertex that follows after the base point of the cut. The boundary vertices lying on the outer circle are labeled $b_{n'_1+1},\dots,b_n$ in the counterclockwise order starting from the first vertex that follows after the base point of the cut.
The transformation $N\mapsto N'$ is better visualized if the network is drawn on a cylinder, instead of an annulus. The boundary circles of a cylinder are identical, and reversing the direction of the cut simply interchanges them. Clearly, the boundary measurements in $N$ and $N'$ are related by $M_{r's'}(t)=M_{rs}(1/t)$ for any $i_r\in I$, $j_s\in J$. Besides, $\Poi_N=\Poi_{N'}$ for any pair of edge weights. Therefore, an expression for $\{M_{pq}(t),M_{\bar p\bar q}(s)\}^1_{N}$ via $M_{p\bar q}(t)$, $M_{p\bar q}(s)$, $M_{\bar p q}(t)$, $M_{\bar p q}(s)$ is transformed to the expression for
$\{M_{p'q'}(t),M_{\bar p'\bar q'}(s)\}^1_{N'}$ via
$M_{p'\bar q'}(t)$, $M_{p'\bar q'}(s)$, 
$M_{\bar p' q'}(t)$, $M_{\bar p' q'}(s)$ by the substitution  $t\mapsto 1/t$ and $s\mapsto 1/s$ in coefficients. For example, let $1\leq i_{\bar p}<j_q<j_{\bar q}\leq n_1<i_p\leq n$. This case is not covered explicitly by relations~(\ref{psre1})--(\ref{psre4}). Consider the network $N'$ obtained from $N$ by reversing the direction of the cut. In this new network one has $1\leq i_{p'}\leq n'_1< i_{\bar p'}<j_q<j_{\bar q'}\leq  n$, so the conditions of Proposition~\ref{PSRE}(ii) are satisfied and~(\ref{psre2}) yields
\[
\{M_{p'q'}(t),M_{\bar p'\bar q'}(s)\}^1_{N'}=
-2M_{p'\bar q'}(t)M_{\bar p' q'}(s)
+\frac{2t}{t-s}M_{p'\bar q'}(t)
\big(M_{\bar p' q'}(t)-M_{\bar p' q'}(s)\big).
\]
Applying the above described rule one gets 
\begin{align*}
\{M_{pq}(t),M_{\bar p\bar q}(s)\}^1_{N}&=
-2M_{p\bar q}(t)M_{\bar p q}(s)
+\frac{2t^{-1}}{t^{-1}-s^{-1}}M_{p\bar q}(t)
\big(M_{\bar p q}(t)-M_{\bar p q}(s)\big)\\
&=-2M_{p\bar q}(t)M_{\bar p q}(s)
-\frac{2s}{t-s}M_{p\bar q}(t)
\big(M_{\bar p q}(t)-M_{\bar p q}(s)\big).
\end{align*}

Finally, reversing the orientation of boundary circles also retains the graph $G$ and changes the labeling of its boundary vertices. Namely, the $n_1$ boundary vertices of $N'$ lying on the outer circle are labeled $b_1,\dots,b_{n_1}$ in the clockwise order starting from the first vertex that follows after the base point of the cut. The boundary vertices lying on the inner circle are labeled $b_{n_1+1},\dots,b_n$ in the counterclockwise order starting from the first vertex that follows after the base point of the cut.
The transformation $N\mapsto N'$ may be visualized as a mirror reflection. Clearly, the boundary measurements in $N$ and $N'$ are related by $M_{r's'}(t)=M_{rs}(1/t)$ for any $i_r\in I$, $j_s\in J$. Besides, $\Poi_N=-\Poi_{N'}$ for any pair of edge weights. Therefore, the transformation of the expressions for the brackets differs from the one for the case of cut reversal by factor $-1$.
For example, let $1\leq j_q<i_p<i_{\bar p}<j_{\bar q}\leq n_1$. This case is not covered explicitly by relations~(\ref{psre1})--(\ref{psre4}). Consider the network $N'$ obtained from $N$ by reversing the orientation of boundary circles. In this new network one has $1\leq i_{p'}< j_{\bar q'}<i_{\bar p'}<j_q\leq  n_1$, so the conditions of Proposition~\ref{PSRE}(iv) are satisfied and~(\ref{psre4}) yields
\[
\{M_{p'q'}(t),M_{\bar p'\bar q'}(s)\}^1_{N'}=
\frac{2t}{t-s}\big(M_{p'\bar q'}(t)-M_{p'\bar q'}(s)\big)
\big(M_{\bar p' q'}(t)-M_{\bar p' q'}(s)\big).
\]
Applying the above described rule one gets 
\begin{align*}
\{M_{pq}(t),M_{\bar p\bar q}(s)\}^1_{N}&=
-\frac{2t^{-1}}{t^{-1}-s^{-1}}
\big(M_{p\bar q}(t)-M_{p\bar q}(s)\big)
\big(M_{\bar p q}(t)-M_{\bar p q}(s)\big)\\
&=\frac{2s}{t-s}\big(M_{p\bar q}(t)-M_{p\bar q}(s)\big)
\big(M_{\bar p q}(t)-M_{\bar p q}(s)\big).
\end{align*}

Elementary, though tedious, consideration of all possible cases reveals that indeed any quadruple $(i_p,j_q,i_{\bar p},j_{\bar q})$ can be reduced by the above three 
transformations to one of the quadruples mentioned in the statement of Proposition~\ref{PSRE}.

It is worth to note that cases (i) and (ii) are not independent. First, they both apply if $1\leq i_p<i_{\bar p}\leq n_1< j_{\bar q}<j_q\leq n$; the  expressions prescribed by~(\ref{psre1}) and~(\ref{psre2}) are distinct, but yield the same result:
\begin{align*}
2M_{p\bar q}(s)M_{\bar p q}(t)
&-\frac{2s}{t-s}
\big(M_{p\bar q}(t)M_{\bar p q}(s)-M_{p\bar q}(s)M_{\bar p q}(t)\big)\\
&=2M_{p\bar q}(t)M_{\bar p q}(s)
-\frac{2t}{t-s}
\big(M_{p\bar q}(t)M_{\bar p q}(s)-M_{p\bar q}(s)M_{\bar p q}(t)\big).
\end{align*}
Besides, the expression for $n_1< i_p<i_{\bar p}< j_{\bar q}<j_q\leq n$ in case (ii) can be obtained from the expressions for $1\leq i_p<i_{\bar p}< j_{\bar q}<j_q\leq n_1$ in case (i) by reversing the direction of the cut. However, we think that the above presentation, though redundant, better emphasizes the underlying symmetries of the obtained expressions.

The proof of relations~(\ref{psre1})--(\ref{psre4}) is similar to the proof of Theorem~3.3 in \cite{GSV3} and is based on the induction on the number of internal vertices in $N$. The key ingredient of the proof is  
%Lemma~\ref{recalc}, which remains true for planar networks in an annulus. 
%To reduce the number of cases to be considered, it is worth to note that part (ii) %of Proposition~\ref{PSRE} follows from part (i) via the above described techniques %in all cases except for $j_{\bar q}=j_q$.
the following straightforward analog of Lemma~3.5 from \cite{GSV3}. 
                                              
Consider an arbitrary boundary vertex $b_i$ (without loss of generality we may assume that $b_i$ lies on the outer circle of the annulus) and       
suppose that the neighbor of $b_i$ is a black vertex $u$.
Denote by $u_+$ the unique vertex in
$G$ such that $(u,u_+)\in E$, and by $u_-$ the neighbor of $u$
distinct from $u_+$ and $b_i$. Create a new network $\hN$ by
deleting $b_i$ and the edge $e_0=(b_i,u)$ from $G$, splitting $u$ into
one new source $b_{i_u}$ and one new sink $b_{j_u}$ placed on the outer circle (so that either $i-1< i_u< j_u< i+1$ or $i-1< j_u<i_u< i+1$) and replacing the
edges $e_+=(u,u_+)$ and $e_-=(u_-,u)$ by new edges $\hat e_+=(b_{i_u},u_+)$ and $\hat e_-=(u_-,b_{j_u})$, see Figure~\ref{fig:blackforkan}.
%Let $\hN$ be the network obtained from $N$ by splitting a black neighbor of a boundary vertex (see Figure~\ref{fig:blackforkan}). 
We may assume without loss of generality that the cut $\rho$ in $N$ does not intersect the edge $e_0$, and hence $\rho$ remains a valid cut in $\hN$. 

\begin{figure}[ht]
\begin{center}
\includegraphics[height=8cm]{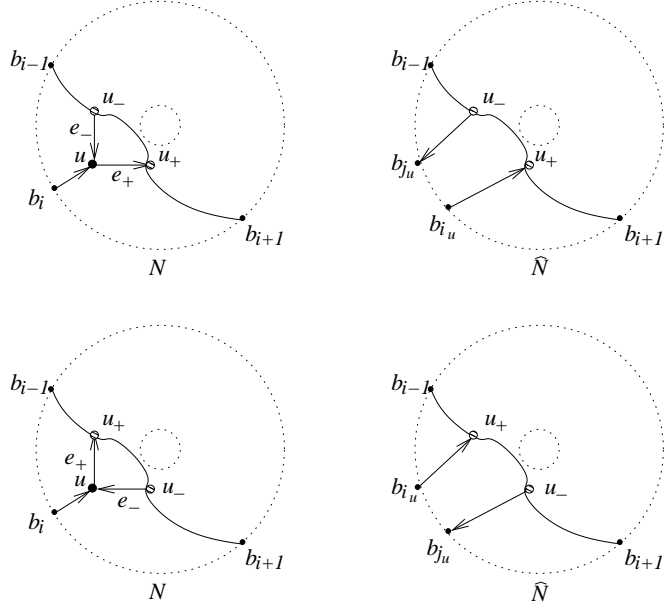}
\caption{Splitting a black vertex: cases $i-1< i_u < j_u < i+1$ 
(upper part) and $i-1 < j_u < i_u < i+1$ (lower part)}
\label{fig:blackforkan}
\end{center}
\end{figure}

\begin{lemma}\label{recalcan}
Boundary measurements in the networks $N$ and $\hN$ are related by
\begin{align*}
M({i_p},j)&=\frac{w_{e_0} w_{e_+}\hM({i_u},j)}{1+w_{e_-}w_{e_+} \hM({i_u},{j_u})},\\
M({i_{\bar p}},j)&=\hM({i_{\bar p}},j)\pm
\frac{w_{e_-}w_{e_+}\hM({i_{\bar p}},{j_u})\hM({i_u},j)}{1+ w_{e_-} w_{e_+}\hM({i_u},{j_u})},\qquad
\bar p\ne p;
\end{align*}
in the second formula above, sign $+$ corresponds to the cases 
\[
i_p-1\prec j_u\prec i_u\prec i_p+1\preceq j \prec i_{\bar p}
\qquad\text{or}\qquad
i_{\bar p}\prec j\preceq i_p-1\prec i_u\prec j_u\prec i_p+1,
\] 
and sign $-$ corresponds to the cases 
\[
i_p-1\prec i_u\prec j_u\prec i_p+1\preceq j \prec i_{\bar p}
\qquad\text{or}\qquad
i_{\bar p}\prec j\preceq i_p-1\prec j_u\prec i_u\prec i_p+1,
\]
where $\prec$ is the cyclic order $\bmod n$.
\end{lemma}

We leave the details of the proof to the interested reader.
\endproof

Denote by $\Poi^2_N$ the member of the 2-parametric
family~(\ref{2parw}),~(\ref{2parb}) corresponding to the case 
%on the image of the boundary measurement map 
$\alpha=\beta=1$. Besides, define $\sigma_\times(i,j,i',j')=\s(i'-i)+\s(j'-j)$; clearly, $\sigma_\times(i,j,i',j')$ is closely related to $s_\times(i,j,i',j')$ defined and studied  in~\cite{GSV3}. 
The bracket induced by $\Poi^2_N$ on the image of the boundary measurement map is completely described by the following statement.

\begin{proposition}\label{PSRE2}
{\rm (i)} Let $1\leq \max\{i_p,i_{\bar p}\}<j_{\bar q}<j_q\leq n$, then
\begin{equation}\label{psre21}
\{M_{pq}(t),M_{\bar p\bar q}(s)\}^2_N=
\sigma_\times(i_p,j_q,i_{\bar p},j_{\bar q})M_{pq}(t)M_{\bar p \bar q}(s)
-2\Gamma_{pq}^{\bar p \bar q}(t,s),
%-2D(M_{pq}(t)M_{\bar p \bar q}(s)),
\end{equation}
where 
\[
\Gamma_{pq}^{\bar p \bar q}(t,s)=\left\{\begin{array}{ll}
0,
&\qquad j_q\leq n_1,\\
%-s \partial/\partial s, 
-sM_{pq}(t)M'_{\bar p \bar q}(s),
&\qquad j_{\bar q}\leq n_1< j_q,\\
%t \partial/\partial t -s \partial/\partial s,
tM'_{pq}(t)M_{\bar p \bar q}(s)-sM_{pq}(t)M'_{\bar p \bar q}(s),
&\qquad \max\{i_p,i_{\bar p}\}\leq n_1< j_{\bar q},\\
%t \partial/\partial t,
tM'_{pq}(t)M_{\bar p \bar q}(s),
&\qquad i_p\leq n_1< i_{\bar p},
\end{array}\right.
\]
and $M'_{pq}$, $M'_{\bar p \bar q}$ are the derivatives of $M_{pq}$ and $M_{\bar p \bar q}$. 

{\rm (ii)} Let $1\leq i_p<j_{\bar q}<i_{\bar p}<j_q \leq n$ and either
$j_{\bar q}\leq n_1<i_{\bar p}$ or $j_q\leq n_1$, then
\begin{equation}\label{psre22}
\{M_{pq}(t),M_{\bar p\bar q}(s)\}^2_N=0.
\end{equation}

{\rm (iii)} Let $1\leq i_p<j_q< i_{\bar p}<j_{\bar q}\leq n$ and either
$j_{q}\leq n_1<i_{\bar p}$ or $j_{\bar q}\leq n_1$, then
\begin{equation}\label{psre23}
\{M_{pq}(t),M_{\bar p\bar q}(s)\}^2_N=0.
\end{equation}
\end{proposition}

\proof 
The proof is similar to the proof of Proposition~\ref{PSRE}. We leave the details to the interested reader.
\endproof

\begin{remark}
It is worth to mention that the bracket induced on $k\times m$ matrices via perfect planar networks in a disk, which was studied in \cite{GSV3}, can be considered as a particular case of~(\ref{psre1})--(\ref{psre4}) (for $\alpha=-\beta=1$) 
or~(\ref{psre21})--(\ref{psre23}) (for $\alpha=\beta=1$). To see this it suffices to consider only networks without edges that intersect the cut $\rho$, and to cut the annulus along $\rho$ in order to get a disk. 
\end{remark}

\subsection{Realization theorem}
To conclude the proof of Theorem~\ref{PSR} we need the following statement.
We say that $F\in\Rat_{k,m}$ is {\it represented by\/} a network $N$ if $F$ belongs to the image of $M_N$.

\begin{theorem}\label{realiz}
For any $F\in\Rat_{k,m}$ there exists a network $N\in \Net_{I_1,J_1,I_2,J_2}$ such that $F$ is represented by $N$.
\end{theorem}

\proof
We preface the proof by the following simple observation concerning perfect planar networks in a disk.

\begin{lemma}\label{transpos}
Let $n=4$, $I=\{1,2\}$, $J=\{3,4\}$. 
There exists a network $N_{\id}\in \Net_{I,J}$ such that the $2\times 2$ identity matrix
%simple transposition matrix 
%\[
%\begin{pmatrix}
%0 & 1\\
%1 & 0
%\end{pmatrix}
%\]
is represented by $N_{\id}$.
%contained in the image of $M_N$.
\end{lemma}

\proof The proof is furnished by the network depicted on Fig.~\ref{fig:transpos}.
The corresponding boundary measurement matrix is given by
\[
\begin{pmatrix}
w_1w_8\big(w_3w_{11}(w_2+w_6w_9w_{10})+w_6w_7w_9\big) & 
w_1w_3w_4(w_2+w_6w_9w_{10})\\
w_1w_6w_8(w_7+w_3w_{10}w_{11}) & w_3w_4w_5w_6w_{10}
\end{pmatrix},
\]
which yields the identity matrix for $w_5=w_{10}=-1$ and $w_i=1$ for $i\ne 5,10$.

\begin{figure}[ht]
\begin{center}
\includegraphics[height=4.5cm]{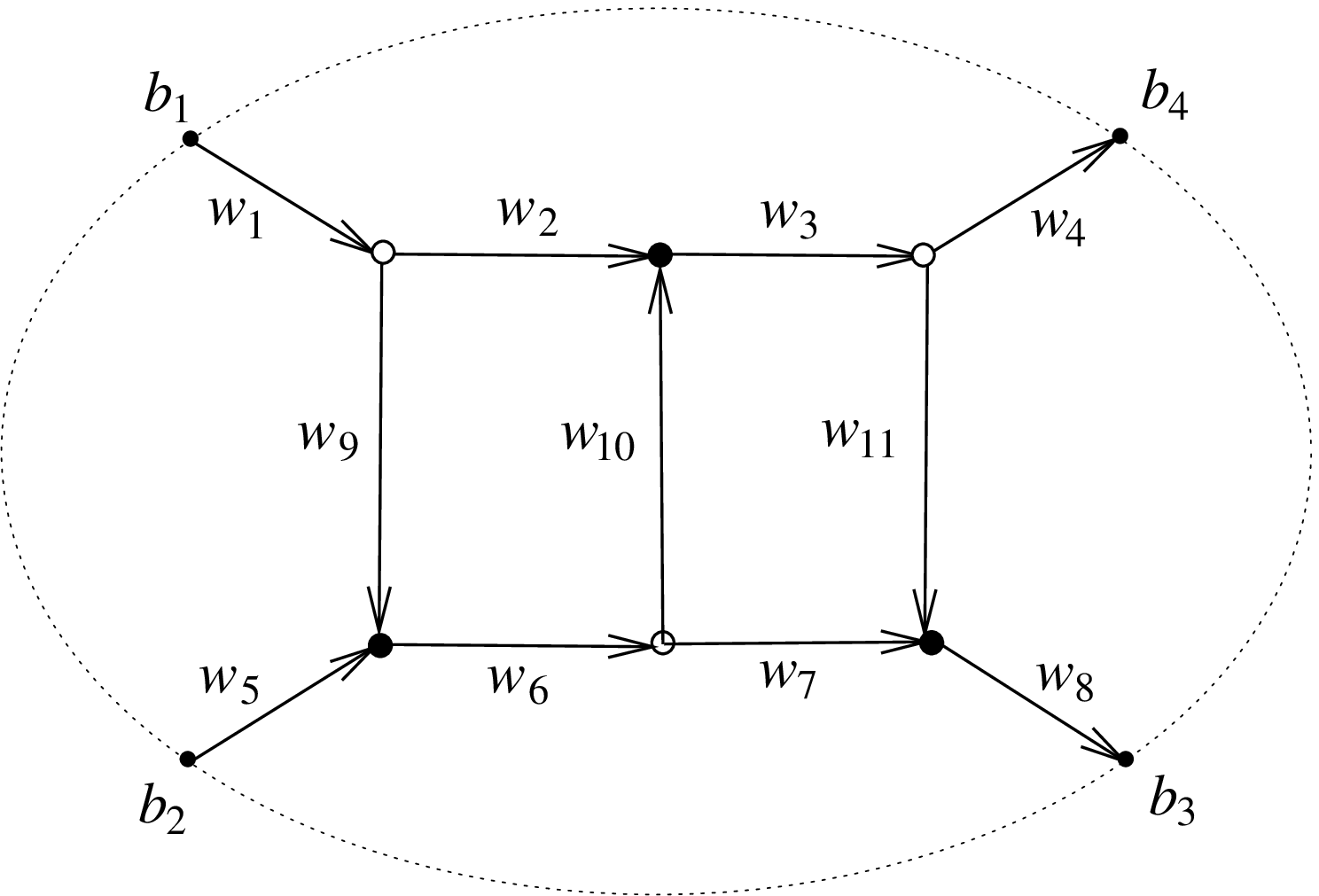}
\caption{Network $N_{\id}$}
\label{fig:transpos}
\end{center}
\end{figure}

\endproof

Effectively, Lemma~\ref{transpos} says that the planarity restriction can be omitted in the proof of Theorem~\ref{realiz}. Indeed, if $F\in\Rat_{k,m}$ is represented by a nonplanar perfect network in an annulus, one can turn it to a planar perfect network in annulus by replacing each intersection by a copy of $N_{\id}$.

In what follows we make use of the concatenation of planar networks in an annulus. Similarly to the case of networks in a disk, the most important particular case of concatenation arises when the sources and the sinks are separated, that is, all sources lie on the outer circle, and all sinks lie on the inner circle. We can concatenate two networks of this kind, one with  $k$  sources and $m$ sinks and another with $m$  sources and
$l$ sinks, by gluing the sinks of the former to the sources of the latter.
More exactly, we glue together the inner circle of the former network and the outer circle of the latter in such a way that the corresponding base points of the cuts 
are identified, and the $i$th sink of the former network is identified with the $(m+1-i)$th source of the latter. The erasure of the common boundary and the identification of edges are performed exactly as in the case of a disk. 

Let us start with representing any rational function $F\in\Rat_{1,1}$ by a network with the only source on the outer circle and the only sink on the inner circle.  

\begin{lemma}\label{1-1realiz}
Any rational function $F\in\Rat_{1,1}$ can be represented by a network $N\in \Net_{1,\varnothing,\varnothing,2}$.
\end{lemma}

\proof
First, if networks $N_1,N_2\in \Net_{1,\varnothing,\varnothing,2}$ represent functions $F_1$ and $F_2$ respectively, their concatenation $N_1\circ N_2\in \Net_{1,\varnothing,\varnothing,2}$ represents $F_1F_2$. 

\begin{figure}[ht]
\begin{center}
\includegraphics[height=5cm]{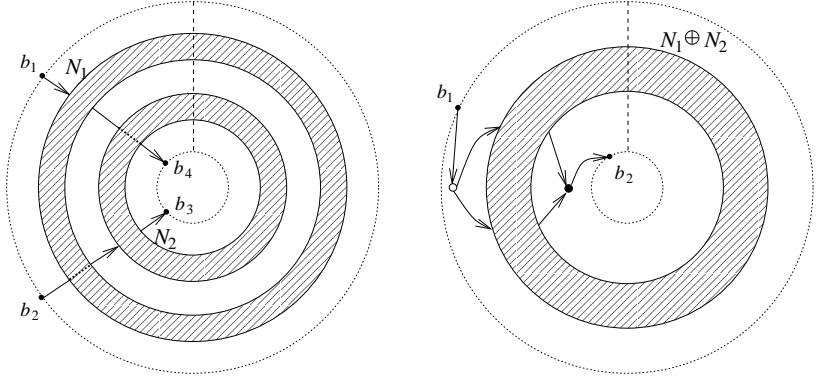}
\caption{The direct sum of two networks (left) and a network representing the sum of two functions (right)}
\label{fig:dirsum}
\end{center}
\end{figure}

Second, define the {\it direct sum\/} $N_1\oplus N_2\in \Net_{[1,2],\varnothing,\varnothing,[3,4]}$ as shown in the left part of Fig.~\ref{fig:dirsum}. 
The shadowed annuli contain networks $N_1$ and $N_2$. The intersections of the dashed parts of additional edges with the edges of $N_1$ and $N_2$ are resolved with the help of $N_{\id}$ (not shown). Note that this direct sum operation is not commutative. Clearly, $N_1\oplus N_2$ represents the $2\times 2$ matrix $\begin{pmatrix}
0 & F_1\\
F_2 & 0
\end{pmatrix}$. The direct sum of networks is used to represent the sum $F_1+F_2$ as shown in the right part of Fig.~\ref{fig:dirsum}. 

Third, if $N\in \Net_{1,\varnothing,\varnothing,2}$ represents $F$, the network shown in Fig.~\ref{fig:netdenom} represents $F/(1+F)$, and, with a simple adjustment of weights, can also be used to represent $-F/(1+F)$. Taking the direct sum with the trivial network representing~$1$, we get a representation for $1/(1+F)$. 

\begin{figure}[ht]
\begin{center}
\includegraphics[height=5cm]{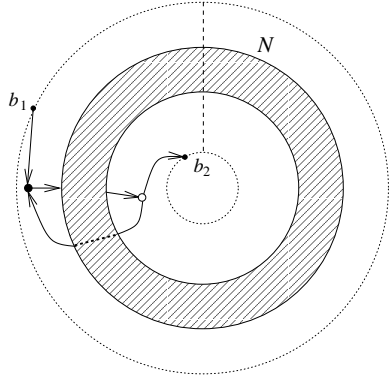}
\caption{Representing $F/(1+F)$}
\label{fig:netdenom}
\end{center}
\end{figure}

Finally, functions $a\lambda^k$ for any integer $k$ can be represented by networks in $\Net_{1,\varnothing,\varnothing,2}$. The cases $k=2$ and $k=-2$ are shown in Fig.~\ref{fig:rep1}. Other values of $k$ are obtained in the same way.

\begin{figure}[ht]
\begin{center}
\includegraphics[height=4cm]{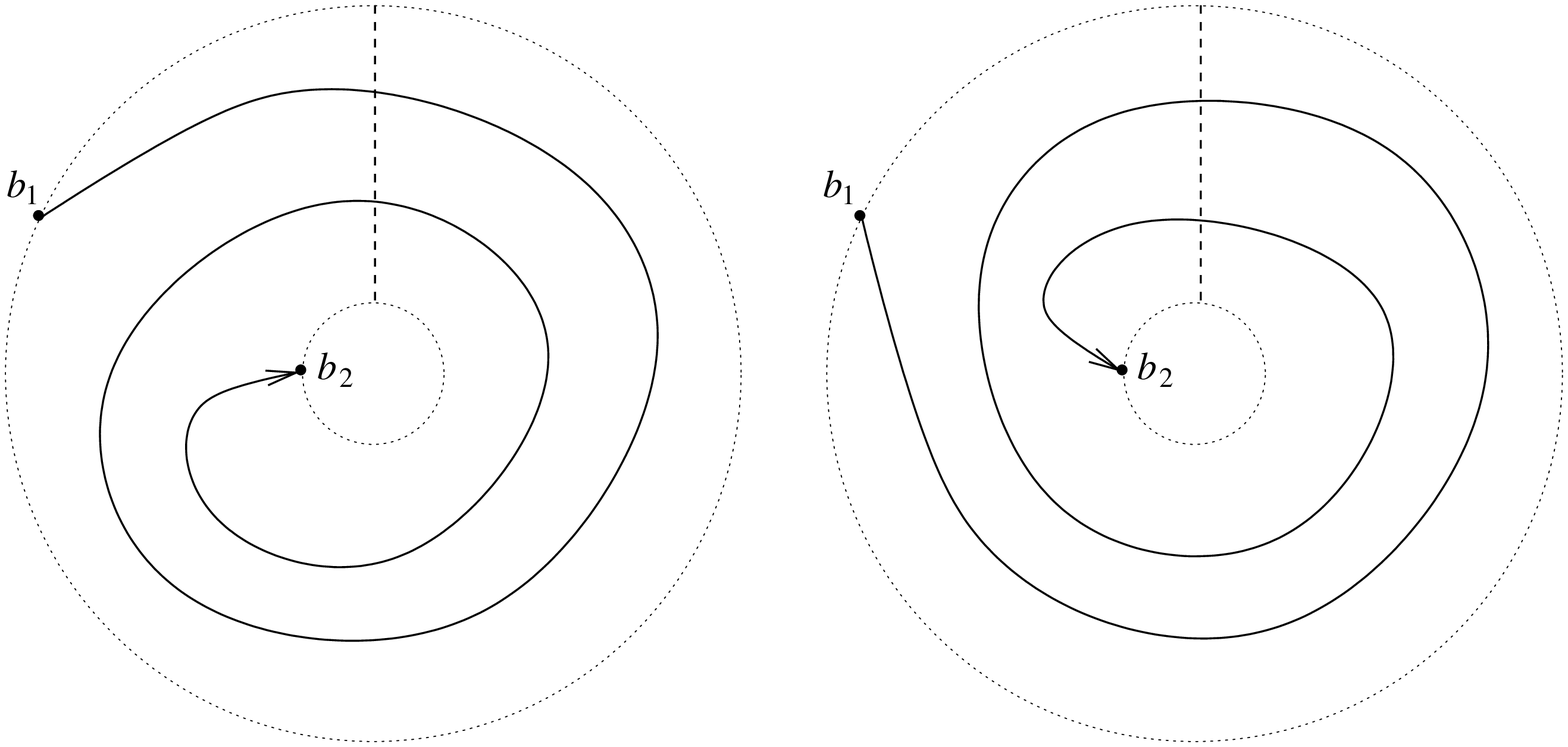}
\caption{Representing $a\lambda^k$ for $k=2$ (left) and $k=-2$ (right)}
\label{fig:rep1}
\end{center}
\end{figure}

We now have all the ingredients for the proof of the lemma. Any rational function $F$ can be represented as $F(\lambda)=\sum_{i=0}^r a_i\lambda^{d+i}/Q(\lambda)$, where $d$ is an integer and $Q$ is a polynomial satisfying $Q(0)=1$. Therefore, it suffices to represent each of the summands, and to use the direct sum construction. Each summand, in its turn, is represented by the concatenation of a network representing $a_i\lambda^{d+i}$ with a network representing $1/Q=1/(1+(Q-1)$. The latter network is obtained as explained above from a network representing $Q-1=\sum_{j=1}^pb_j\lambda^{j}$ via the direct sum construction.
\endproof

To get an analog of Lemma~\ref{1-1realiz} for networks with the only source and the only sink on the outer circle, one has to use once again Lemma~\ref{transpos}.  

\begin{lemma}\label{2-0realiz}
Any rational function $F\in\Rat_{1,1}$ can be represented by a network $N\in \Net_{1,2,\varnothing,\varnothing}$.
\end{lemma}

\proof  Such a representation is obtained from the one constructed in the proof of Lemma~\ref{1-1realiz} by replacing the edge incident to the sink with a new edge sharing the same tail. The arising intersections, if any, are resolved with the help of $N_{\id}$. For example, representation of $a(1+b\lambda)^{-1}$ obtained this way is shown in Fig.~\ref{fig:rep6} on the right. It makes use of the network $N_{\id}$ described in Lemma~\ref{transpos}; the latter is shown in thin lines inside a dashed circle.
Note that the network on the left, which represents $a(1+b\lambda)^{-1}$ in $\Net_{1,\varnothing,\varnothing,2}$, is not the one built in the proof of Lemma~\ref{1-1realiz}.

\begin{figure}[ht]
\begin{center}
\includegraphics[height=5cm]{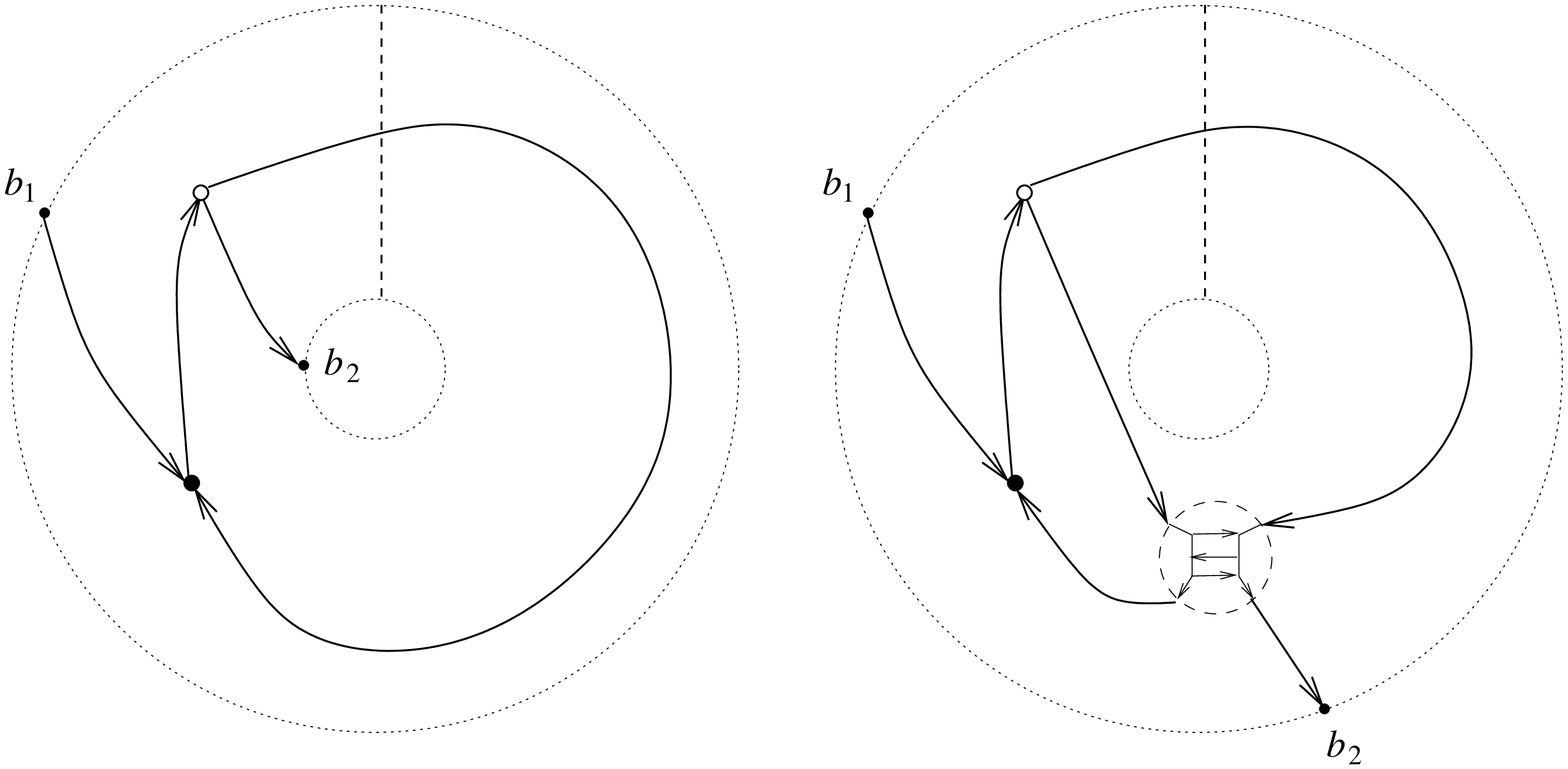}
\caption{Representing $a(1+b\lambda)^{-1}$ by a network in
$\Net_{1,\varnothing,\varnothing,2}$ (left) and 
 $\Net_{1,2,\varnothing,\varnothing}$ (right)}
\label{fig:rep6}
\end{center}
\end{figure}

\endproof

Representation of rational functions by networks in $\Net_{2,1,\varnothing,\varnothing}$ and $\Net_{\varnothing,1,2,\varnothing}$ is obtained in a similar way. In the latter case one has to replace also the edge incident to the source with a new one sharing the head.

The next step is to prove Theorem~\ref{realiz} in the case when all sources lie on the outer circle and all sinks lie on the inner circle.

\begin{lemma}\label{seprealiz}
For any rational matrix $F\in\Rat_{k,m}$ there exists a network $N\in \Net_{[1,k],\varnothing,\varnothing,[k+1,k+m]}$ such that $F$ is represented by $N$.
\end{lemma}

\proof
First of all, we represent $F$ as $F=A\widetilde{F}B$, where $A=\{a_{ij}\}$ is the $k\times km$ constant matrix given by
\[
a_{ij}=\left\{\begin{array}{ll}
1, \quad & \text{if $(k-i)m<j\leq (k-i+1)m$},\\
0, \quad & \text{otherwise},
\end{array}\right.
\]
$\widetilde{F}$ is the $km\times km$ diagonal matrix 
\[
\widetilde{F}=\diag\{F_{km},F_{k,m-1},\dots,F_{k1},F_{k-1,m},\dots, F_{11}\},
\]
and $B$ is the $km\times m$ constant matrix
\[
B=\begin{pmatrix}
W_0\\
\vdots\\
W_0
\end{pmatrix}
\]
with $W_0=(\delta_{i,m+1-j})_{i,j=1}^m$.
Similarly to the case of networks in a disk, the  concatenation of networks representing matrices $F_1$ and $F_2$ produces a network representing $F_1W_0F_2$.
%(assuming that the number of sinks in the first network and the number of sources %in the second network equal $m$). 
Therefore, in order to get $F$ as above, we have to represent matrices $A$, $W_0\widetilde{F}$ and $W_0B$.

The first representation is achieved trivially as the disjoint union of networks representing the $1\times m$ matrix $(1\ 1\ \dots 1)$; in fact, since $A$ is constant, it can be represented by a network in a disk.
The second representation is obtained as the direct sum of networks representing each of $F_{ij}$. Finally, the third representation can be also achieved by a network in a disk, via a repeated use of the network $N_{\id}$.

Observe that in order to represent a $k\times m$ matrix we have to use intermediate matrices of a larger size.
\endproof

To complete the proof of Theorem~\ref{realiz} we rely on Lemma~\ref{seprealiz} and use the same idea of replacing edges incident to boundary vertices as in the proof of Lemma~\ref{2-0realiz}.

So, Theorem~\ref{realiz} has been proved, and hence the proof of Theorem~\ref{PSR} is completed.
\endproof

As we already mentioned in the proof, the 2-parametric family of Poisson brackets on $\Rat_{k,m}$ induced by 
$(\alpha-\beta)\Poi_N^1+(\alpha+\beta)\Poi_N^2$, where $\Poi_n^1$ and $\Poi_n^2$ are described in Propositions~\ref{PSRE} and~\ref{PSRE2}, respectively,
is denoted $\Poi_{I_1,J_1,I_2,J_2}$. 

\subsection{Recovering the trigonometric R-matrix bracket on $\Rat_{k,k}$}

As an application of results obtained in Section~\ref{PSonrat}, consider the set 
%Let us consider a situation when $ n = 2k$ and a network $N$ belongs to  
$N_{[1,k],\varnothing,\varnothing,[k+1,2k]}$ of perfect networks with $k$ sources on the outer circle and $k$ sinks on the inner circle. Clearly, in this case  
the boundary measurement map takes $\CC_N$ to $\Rat_{k,k}$. Just as we did in the proof of Theorem~\ref{realiz}, we can replace  $M_N$ with $A_N= M_N W_0$ and observe that the concatenation $N$
of networks $N_1, N_2 \in N_{[1,k],\varnothing,\varnothing,
[k+1,2k]}$ leads to $A_N = A_{N_1} A_{N_2}$. We would like to take a closer look at  the bracket
$\Poi_N^1$ in this case.

First, recall (see, e.g. \cite{FT}) that the space $\Rat_{k,k}$ can be
equipped with an R-matrix  (Sklyanin)  Poisson bracket
\begin{equation}
\{A(t), A(s)\} = \left [ R(t, s), A(s)\otimes A(t) \right ],
\label{ansklyanin}
\end{equation}
where the left-hand side should be understood as
\[
\{A(t), A(s)\}_{p\bar p}^{q\bar q}= \{a_{pq}(t), a_{\bar p \bar q}(s)\}
\]
and the R-matrix $R(t, s)$ is an operator acting in $\mathbb{R}^{2k}\otimes  \mathbb{R}^{2k}$ that depends on parameters $t, s$  and solves  the classical Yang-Baxter equation. Of interest to us is
a bracket~(\ref{ansklyanin}) that corresponds to the so-called {\em trigonometric R-matrix} \cite{BD}
\begin{equation}{\displaystyle
R(t, s)=
%R_n(t, s)= 
\frac{t+s}{s-t}\sum_{k=1}^n E_{kk}\otimes E_{kk}\ 
+\ \frac{2}{s-t}\sum_{1 \le l < m \le 2k}
\left ( t E_{lm}\otimes E_{ml} + s E_{ml}\otimes E_{lm}\right ).
}
\label{trig-matrix}
\end{equation}

Bracket~(\ref{ansklyanin}),~(\ref{trig-matrix}) can be re-written in terms of matrix entries of $A(t)$ as follows (we only list non-zero brackets):
for $p<\bar p$ and $q< \bar q$,
\begin{equation}
{\displaystyle
\{a_{pq}(t), a_{\bar p \bar q}(s)\} =
2 \frac{t a_{p\bar q}(s) a_{\bar p q}(t)
-s a_{p\bar q}(t) a_{\bar p q}(s)}{t-s}.
}
\label{brack_ij_ab}
\end{equation}

\begin{equation}
{\displaystyle
\{a_{p\bar q}(t), a_{\bar p q}(s)\} = 2 t
 \frac{a_{pq}(s) a_{\bar p \bar q}(t)- a_{pq}(t) a_{\bar p \bar q}(s)}{t-s}.
}
\label{brack_ib_aj}
\end{equation}

\begin{equation}
{\displaystyle
\{a_{pq}(t), a_{p \bar q}(s)\} =
 \frac{(t+s) a_{p\bar q}(s) a_{pq}(t)- 2s a_{p\bar q}(t) a_{pq}(s)}{t-s}.
}
\label{brack_ij_ib}
\end{equation}

\begin{equation}
{\displaystyle
\{a_{pq}(t), a_{\bar p q}(s)\} =
 \frac{2 t a_{pq}(s) a_{\bar p q}(t)-
 (t+s) a_{pq}(t) a_{\bar p q}(s)}{t-s}.
}
\label{brack_ij_aj}
\end{equation}

It is now straightforward to check that for $N \in N_{[1,k],\varnothing,\varnothing,
[k+1,2k]}$, the Poisson algebra satisfied by the entries of $A_N$ coincides with
that of the Sklyanin bracket~(\ref{ansklyanin}),~(\ref{trig-matrix}). More exactly,
relations~(\ref{brack_ij_ab}) and~(\ref{brack_ij_ib}) are equivalent
to~(\ref{psre1}) with $\Phi_{pq}^{\bar p \bar q}(t,s)$ calculated according to the third case, while~(\ref{brack_ib_aj}) and~(\ref{brack_ij_aj})  are equivalent
to~(\ref{psre2}) with $\Psi_{pq}^{\bar p \bar q}(t,s)$ calculated according to the first case. Finally, the brackets that vanish identically, correspond exactly to the situations listed in the second case in~(\ref{psre3}) and in the third case
in~(\ref{psre4}).

To summarize, we obtained the following statement.

\begin{theorem}\label{sklyaninbr}
For any $N \in N_{[1,k],\varnothing,\varnothing,
[k+1,2k]}$ and any choice of parameters $\alpha_{ij}, \beta_{ij}$ in~(\ref{6parw}),~(\ref{6parb}) such that 
$\alpha=1$ and $\beta=-1$ in~(\ref{cond}), the map $A_N : (\mathbb{R}\setminus 0)^{|E|} \to \Rat_{k,k}$ is Poisson
with respect to the Sklyanin bracket~(\ref{ansklyanin}) associated with the R-matrix~(\ref{trig-matrix}).
\end{theorem}

\begin{remark}{\rm 
Equations~(\ref{brack_ij_ab})--(\ref{brack_ij_aj}) can be also used to define a Poisson bracket in the ``rectangular'' case of $\Rat_{k_1,k_2}$. In this case, a concise description~(\ref{ansklyanin}) of the bracket should be modified as follows:
%\begin{equation}
\[
\{A(t), A(s)\} = 
R_{k_1}(t, s) \left (A(t)\otimes A(s)\right ) -  
\left (A(t)\otimes A(s)\right )  R_{k_2}(t, s),
%\label{ansklyanin_rec}
%\end{equation}
\]
where $R_{k_i}(t, s)$ denotes the R-matrix~(\ref{trig-matrix}) acting
in $\mathbb{R}^{k_i} \otimes \mathbb{R}^{k_i}$.
}
\end{remark}

\section[Grassmannian boundary measurement map]
{Poisson properties of the Grassmannian boundary
measurement map}\label{PSongrloop}

\subsection{Grassmannian boundary measurement map and path reversal}\label{PathRev}
Let $N\in\Net_{I_1,J_1,I_2,J_2}$ be a perfect planar network in an annulus. Similarly to the case of a disk, we are going to provide a Grassmannian interpretation of the boundary measurement map defined by $N$.  To this end, we extend the boundary measurement matrix $M_N$ to a $k\times n$ matrix $\wX_N$ as follows: 

(i) the $k\times k$ submatrix of $\wX_N$ formed
by $k$ columns indexed by $I=I_1\cup I_2$ is the identity matrix $\one_k$; 

(ii) for $p\in[1,k]$ and $j=j_q\in J$, the 
$(p,j)$-entry of $\wX_N$ is $m^I_{pj}= (-1)^{s(p,j)} M_{p q}$, where $s(p,j)$ is the number of elements in $I$ lying strictly between $\min\{i_p,j\}$ and $\max\{i_p,j\}$ in the linear ordering; note that the sign is selected in such a way that the minor
$(\wX_N)_{[1,k]}^{I(i_p\to j)}$ coincides with $M_{p q}$, where $I(i_p\to j)=(I\setminus i_p)\cup j$. 

We will view $\wX_N$ as a matrix representative of an element $X_N$ in the space $LG_k(n)$ of rational functions $X\:\R\to G_k(n)$. The latter space is called the {\it space of Grassmannian loops\/}, and the 
corresponding rational map $X_N\: \EE_N\to LG_k(n)$ is called the {\it Grassmannian boundary measurement map}. 

 Given a network $N$ and a simple path $P$ from a source $b_i$ to a sink $b_j$ in $N$, we define the {\it reversal\/} of $P$ as follows: for every edge $e\in P$, change its direction and replace its weight $w_e$ by $1/w_e$;  equivalently, the modified weight $\bar w_e$ is replaced by $1/\bar w_e$. Clearly, after the reversal of $P$ all vertices preserve their color.

Denote by $N^P$ the network obtained from $N$ by the reversal of $P$, and by $R^P$ the corresponding {\it path reversal map\/} $\EE_N\to\EE_{N^P}$. Besides, put $t^P=1$ if both
endpoints of $P$ belong to the same boundary circle, and $t^P=-1$ otherwise. Define two maps from
$LG_k(n)$ to itself: $S_1$ is the identity map, while $S_{-1}$ takes any $X(t)\in LG_k(n)$ to
$X(-t)$. Our next goal is to prove that the path reversal map 
for paths not intersecting the cut
commutes with the Grassmannian boundary measurement map up to $S_{t^P}$. 

\begin{theorem}\label{pathrev}
Let $P$ be a simple path from a source $b_i$ to a sink $b_j$ in $N$ such that $M(i,j)$ does not vanish identically and $P$ does not intersect the cut. Then  
\[
S_{t^P}\circ X_N=X_{N^P}\circ R^P.
\]
\end{theorem}

\proof Let $I$ be the index set of the sources in $N$. The statement of the Theorem is equivalent to the equality $S_{t^P}(x_K)x_I^P=x_K^P$ 
for any subset $K\subset [1,n]$ of size $k$. Here and in what follows the superscript $P$ means that the corresponding value is related to the network $N^P$. 
%As it was explained above, 
The signs of the elements $m^I_{pj}$ are chosen in such a way that
$x_I^P=M^P(j,i)$, so we have to prove
\begin{equation}\label{pathreveq}
S_{t^P}(x_K)M^P(j,i)=x_K^P.
\end{equation}

The proof  of~(\ref{pathreveq}) relies on the induction on the number of inner vertices in $N$.

Let us start with the case when $N$ does not have inner vertices. In this case it suffices to 
prove~\eqref{pathreveq} with $K=I(i_r\to l)$ for all edges $e=(b_{i_r},b_l)$. 
Assume first that the intersection 
index of each edge with the cut $\rho$ equals $\pm1$ or $0$; consequently, after a suitable isotopy, 
each edge either intersects $\rho$  exactly once, or does not intersect it at all. Let $e^*=(b_i,b_j)$ 
be the edge to be reversed, and hence $\ind(e^*)=0$. If both $b_i$ and $b_j$ belong to the same 
boundary circle, then exactly the following two cases are prohibited: 
\begin{align*}
&\min\{i_r,l\}<\min\{i,j\}<\max\{i_r,l\}<\max\{i,j\},\\ 
&\min\{i,j\}<\min\{i_r,l\}<\max\{i,j\}<\max\{i_r,l\}. 
\end{align*}
Consequently,  
reversing $e^*$ does not change $(-1)^{s(r,l)}$, which corresponds to the map $S_1$. If $b_i$ and $b_j$ belong 
to distinct boundary circles, then the above two cases  are prohibited whenever $e$ 
does not intersect the cut. If $e$ intersects the cut then the above two cases are the only possibilities.
Consequently, reversing $e^*$ does not change $(-1)^{s(r,l)}$ for the edges not intersecting the cut and reverses 
it for the edges intersecting the cut, which corresponds to the map $S_{-1}$. 

It remains to lift the restriction on the intersection index of edges with $\rho$. If there exists an edge $e'$ such that $|\ind (e')|>1$ then the endpoints of $e'$ belong to distinct boundary circles, and for any other edge with the endpoints on distinct boundary circles, the intersection index with $\rho$ does not vanish. Consequently, only edges with the endpoints on the same boundary circle can be reversed, and the above reasoning applies, which leads to $S_1$.

Let now $N$ have inner vertices, and assume that the first inner vertex $v$ on $P$ is white. Denote by $e$ and $e'$ the first two edges of $P$, and by $e''$ the third edge incident to $v$. 
In what follows we assume without loss of generality that the cut in $N$ does not intersect $e$.
To find $M^P(j,i)$ consider the network  $\widehat{N^P}$ that is related to $N^P$ exactly in the same way as the network $\widehat{N}$ defined immediately before Lemma~\ref{recalcan} is related to $N$. Similarly to the first relation in Lemma~\ref{recalcan}, we find
\[
M^P(j,i)=\frac{ w^P_{e} w^P_{e'}\widehat{M^P}(j,j_v)}
{1+ w^P_{e'} w^P_{e''}\widehat{M^P}(i_v,j_v)}.
\]
Taking into account that $w^P_{e}=1/w_{e}$, $w^P_{e'}=1/w_{e'}$, $w^P_{e''}=w_{e''}$, we finally get
\begin{equation}\label{xipwhite}
M^P(j,i)=
\frac{\widehat{M^P}(j,j_v)}{w_{e}w_{e'}+w_{e}w_{e''} \widehat{M^P}(i_v,j_v)}.
\end{equation}

To find $x_K^P$ we proceed as follows.

\begin{lemma}\label{xkpwhite} 
Let the first inner vertex $v$ of $P$ be white, then
\[
x_K^P=\begin{cases} 
\dfrac{w_{e'}(\widehat{x^P})_{K\cup i_v}+w_{e''}(\widehat{x^P})_{K\cup j_v}}
{w_{e'}+w_{e''}\widehat{M^P}(i_v,j_v)}
\qquad &\text{if $i\notin K$}\\
\\
\dfrac{(\widehat{x^P})_{K(i\to j_v)\cup i_v}} 
{w_{e}w_{e'}+w_{e}w_{e''}\widehat{M^P}(i_v,j_v)}\qquad&\text{if $i\in K$}.
\end{cases}
\]
\end{lemma}

\proof The proof utilizes explicit formulas (similar to those provided by Lem\-ma~\ref{recalcan}) that relate boundary measurements in the networks $N^P$ and $\widehat{N^P}$. What is important, the sign $\pm$ in the second formula in Lemma~\ref{recalcan} and the sign $(-1)^{s(p,j)}$ defined at the beginning of this Section interplay in such a way that any submatrix of $\wX_{N^P}$ is the sum of the corresponding submatrix of $\wX_{\widehat{N^P}}$ and a submatrix of the rank~$1$ matrix that is equal to the tensor product of the $i$th column of $\wX_{N^P}$ and  the $j_v$th row of $\wX_{\widehat{N^P}}$.
\endproof

To find $x_K$, 
%we consider the network $\widetilde{N}$ defined as in the proof of Proposition~\ref{sfree}.
create a new network $\wN$ by
deleting $b_i$ and the edge $e$ from $G$, splitting $v$ into
$2$ sources $b_{i'_v}, b_{i''_v}$ (so that either $i-1< i'_v< i''_v< i+1$ or $i-1< i''_v< i'_v< i+1$) and replacing
the edges $e'=(v,v')$ and $e''=(v,v'')$ by $(b_{i'_v},v')$ and
$(b_{i''_v},v'')$, respectively. 
%(see Figure~\ref{fig:whitefork}).

\begin{lemma}\label{xkwhite} 
Let the first inner vertex $v$ of $P$ be white, then
\[
x_K=\begin{cases} 
w_{e}w_{e'} \widetilde{x}_{K\cup i''_v}+w_{e}w_{e''}\widetilde{x}_{K\cup i'_v}
\qquad&\text{if $i\notin K$}\\
\\
\widetilde{x}_{K(i\to i'_v)\cup i''_v} \qquad&\text{if $i\in K$}.
\end{cases}
\]
\end{lemma}

\proof The proof is a straightforward computation.
\endproof

By~(\ref{xipwhite}) and Lemmas~\ref{xkpwhite} and~\ref{xkwhite}, relation (\ref{pathreveq}) boils down to 
\[
S_{t^P}(w_{e'}\widetilde{x}_{K\cup i''_v}+w_{e''}\widetilde{x}_{K\cup i'_v})
\widehat{M^P}(j,j_v)=
w_{e'}(\widehat{x^P})_{K\cup i_v}+w_{e''}(\widehat{x^P})_{K\cup j_v}
\]
for $i\notin K$ and
\[
S_{t^P}(\widetilde{x}_{K(i\to i'_v)\cup i''_v})\widehat{M^P}(j,j_v)=
(\widehat{x^P})_{K(i\to j_v)\cup i_v}
\]
for $i\in K$. To prove these two equalities, we identify $b_{i'_v}$ with $b_{j_v}$ and $b_{i''_v}$ with $b_{i_v}$. Under this identification we have $\widetilde{N}^{\widetilde{P}}=\widehat{N^P}$, where $\widetilde{P}$ is the path from $b_{i'_v}$ to $b_j$ in $\widetilde{N}$ induced by $P$. Observe that $\widetilde{N}$ has less inner vertices than $N$, and that the index set of the sources in $\widetilde{N}$ is $I(i\to i'_v)\cup i''_v$. Therefore, by the induction hypothesis, 
\begin{equation}\label{indhyp}
S_{t^{\widetilde{P}}}(\widetilde{x}_{\widetilde{K}}) \widetilde{x}^{\widetilde{P}}_{I(i\to i'_v)\cup i''_v}=\widetilde{x}_{\widetilde{K}}^{\widetilde{P}}
\end{equation} 
for any $\widetilde{K}$ of size $k+1$. Besides, $\widetilde{x}^{\widetilde{P}}_{I(i\to i'_v)\cup i''_v}=
\widetilde{M}^{\widetilde{P}}(j,i'_v)=\widehat{M^P}(j,j_v)$ and $t^P=t^{\widetilde{P}}$.
Therefore, 
using~(\ref{indhyp}) for $\widetilde{K}= K\cup i'_v=K\cup j_v$, 
$\widetilde{K}= K\cup i''_v=K\cup i_v$ and $\widetilde{K}=K(i\to i'_v)\cup i''_v=K(i\to j_v)\cup i_v$ we get both equalities above.

Assume now that the first inner vertex $v$ on $P$ is black.  Denote by $e$ and $e'$ the first two edges of $P$, and by $e''$ the third edge incident to $v$. 
To find $M^P(j,i)$, consider the network  $\widetilde{N^P}$ similar to the one defined immediately before Lemma~\ref{xkwhite}, the difference being that the two new boundary vertices $j'_v$ and $j''_v$ are sinks rather than sources. Clearly,
\begin{equation}\label{xipblack}
M^P(j,i)=
\frac1{w_{e}w_{e'}}\left(\widetilde{M^P}(j,j'_v)+w_{e''}w_{e'}
\widetilde{M^P}(j,j''_v)\right).
\end{equation}

To find $x_K^P$ we proceed as follows.

\begin{lemma}\label{xkpblack} 
Let the first inner vertex $v$ of $P$ be black, then
\[
x_K^P=\begin{cases} (\widetilde{x^P})_K
\qquad\qquad\qquad&\text{if $i\notin K$}\\
\\
\dfrac1{w_{e}w_{e'}}\left((\widetilde{x^P})_{K(i\to j'_v)}+
w_{e''}w_{e'}(\widetilde{x^P})_{K(i\to j''_v)}\right)\qquad&\text{if $i\in K$}.
\end{cases}
\]
\end{lemma}

\proof The proof is a straightforward computation.
\endproof

To find $x_K$ we consider the network $\widehat{N}$ defined immediately before 
Lemma~\ref{recalcan}.

\begin{lemma}\label{xkblack} 
Let the first inner vertex $v$ of $P$ be black, then
\[
x_K=\begin{cases} 
\dfrac{w_{e}w_{e'}\widehat{x}_K}{1+w_{e''}w_{e'} \widehat{M}(i_v,j_v)}
\qquad&\text{if $i\notin K$}\\
\\
\dfrac{\widehat{x}_{K(i\to i_v)}+w_{e''}w_{e'}\widehat{x}_{K(i\to j_v)}}
{1+w_{e''}w_{e'}\widehat{M}(i_v,j_v)}
\qquad&\text{if $i\in K$}.
\end{cases}
\]
\end{lemma}

\proof The proof is similar to the proof of Lemma~\ref{xkpwhite}.
\endproof

By~(\ref{xipblack}) and Lemmas~\ref{xkpblack} and~\ref{xkblack}, relation (\ref{pathreveq}) boils down to 
\[
S_{t^P}(\widehat{x}_K)\widetilde{M^P}(j,j'_v)\left(1+w_{e''}w_{e'}
\dfrac{\widetilde{M^P}(j,j''_v)}{\widetilde{M^P}(j,j'_v)}\right)
=(\widetilde{x^P})_KS_{t^P}\left(1+w_{e''}w_{e'}\widehat{M}(i_v,j_v)\right)
\]
for $i\notin K$ and
\begin{multline*}
S_{t^P}\left(\widehat{x}_{K(i\to i_v)}+w_{e''}w_{e'}\widehat{x}_{K(i\to j_v)}\right)
\widetilde{M^P}(j,j'_v)\left(1+w_{e''}w_{e'}
\dfrac{\widetilde{M^P}(j,j''_v)}{\widetilde{M^P}(j,j'_v)}\right)\\
=\left((\widetilde{x^P})_{K(i\to j'_v)}+
w_{e''}w_{e'}(\widetilde{x^P})_{K(i\to j''_v)}\right)S_{t^P}\left(1+w_{e''}w_{e'} \widehat{M}(i_v,j_v)\right)
\end{multline*}
for $i\in K$. To prove these two equalities, we identify $b_{j'_v}$ with $b_{i_v}$ and $b_{j''_v}$ with $b_{j_v}$. Under this identification we have $\widehat{N}^{\widehat{P}}=\widetilde{N^P}$, where $\widehat{P}$ is the path from $b_{i_v}$ to $b_j$ in $\widehat{N}$ induced by $P$. Observe that $\widehat{N}$ has less inner vertices than $N$, and that the index set of the sources in $\widehat{N}$ is $I(i\to i_v)$. Therefore, by the induction hypothesis, 
\begin{equation}\label{indhypb}
S_{t^P}(\widehat{x}_{\widehat{K}}) \widehat{x}^{\widehat{P}}_{I(i\to i_v)}=\widehat{x}_{\widehat{K}}^{\widehat{P}}
\end{equation} 
for any $\widehat{K}$ of size $k$. Taking into account that  $\widehat{x}_{I(i\to j_v)}=\widehat{M}(i_v,j_v)$,
\[
\widehat{x}^{\widehat{P}}_{I(i\to i_v)}=
\widehat{M}^{\widehat{P}}(j,j_v)=\widetilde{M^P}(j,j''_v), \qquad 
\widehat{x}^{\widehat{P}}_{I(i\to j_v)}=
\widehat{M}^{\widehat{P}}(j,i_v)=\widetilde{M^P}(j,j'_v),
\]
and using~(\ref{indhypb}) for $\widehat{K}= K$, 
$\widehat{K}= K(i\to i_v)=K(i\to j'_v)$, $\widehat{K}=K(i\to j_v)=K(i\to j''_v)$ and $\widehat{K}=I(i\to j_v)$ we get both equalities above.
\endproof

\begin{remark} Theorem~\ref{pathrev} is proved in~\cite{Postnikov} for networks in a disk. Observe that in this case $t^P$ vanishes identically, and hence $X_N$ and $R^P$ always commute.
\end{remark}

\subsection{Induced Poisson structures on $LG_k(n)$}\label{indpoissonl} 
Consider a subspace $LG_k^I(n)\subset LG_k(n)$ consisting of all $X\in LG_k(n)$ such that the Pl\"ucker coordinate $x_I$ does not vanish identically; clearly, $X_N\in LG_k^I(n)$. Therefore, we can identify $LG_k^I(n)$ with the space $\Rat_{k,m}$ equipped with the 2-parametric family of Poisson brackets $\Poi_{I_1,J_1,I_2,J_2}$.
 
The following result says that for any fixed $n_1=|I_1|+|J_1|$, 
the families of brackets $\Poi_{I_1,J_1,I_2,J_2}$ on different subspaces $LG^I_k(n)$ can be glued together to form the unique 2-parametric family of Poisson brackets on $LG_k(n)$ that makes all maps $X_N$ Poisson.

\begin{theorem}\label{PSLGr}
{\rm(i)} For any fixed $n_1$, $0\leq n_1\leq n$, and any choice of parameters $\alpha$ and $\beta$ there exists a unique Poisson bracket $\P^{n_1}_{\alpha,\beta}$ on $LG_k(n)$ such that for any network $N$ with $n_1$ boundary vertices on the outer circle, $n-n_1$ boundary vertices on the inner circle, $k$ sources, $n-k$ sinks and weights defined by~(\ref{connect}),  the map $X_N\: (\R\setminus 0)^{|E|}\to LG_k(n)$ is Poisson provided the parameters $\alpha_{ij}$ and $\beta_{ij}$ defining the bracket $\{\cdot,\cdot\}_N$ on $(\R\setminus 0)^{|E|}$ satisfy relations~(\ref{cond}).

{\rm(ii)} For any $I\subset [1,n]$, $|I|=k$,  and any $n_1$, $0\leq n_1\leq n$, the restriction of $\P^{n_1}_{\alpha,\beta}$ to the subspace $LG^I_k(n)$ coincides with the bracket $\{\cdot,\cdot\}_{I_1,J_1,I_2,J_2}$ 
with $I_1=I\cap [1,n_1]$, $J_1=[1,n_1]\setminus I_1$, $I_2=I\setminus I_1$, $J_2=[n_1+1,n]\setminus I_2$.
\end{theorem}

\proof
This result is an analog of Theorem~4.3 proved in \cite{GSV3}, and one may attempt to prove it in a similar way. The main challenge in implementing such an approach is to check that the Poisson structures defined for two distinct subspaces $LG_k^I(n)$ and $LG_k^{I'}(n)$ coincide on the intersection $LG_k^I(n)\cap LG_k^{I'}(n)$. For the case of networks in an annulus, the direct check becomes too cumbersome. We suggest to bypass this difficulty in the following way. 

 Assume first that $|I\cap I'|=k-1$ and take $i\in I\setminus I'$, $j\in I'\setminus I$. 
Denote by  $\Net_{I_1,J_1,I_2,J_2}^{ij}$ the set of networks in $\Net_{I_1,J_1,I_2,J_2}$ satisfying the following two conditions: $M(i,j)$ does not vanish identically and there exists a path from $b_i$ to $b_j$ that does not intersect
the cut. The set $\Net_{I'_1,J'_1,I'_2,J'_2}^{ji}$ is defined similarly, with the roles of $i$ and $j$ interchanged.
Clearly, the path reversal introduced in Section~\ref{PathRev} establishes a bijection between  $\Net_{I_1,J_1,I_2,J_2}^{ij}$ 
%of networks in $\Net_{I_1,J_1,I_2,J_2}$ containing a path from $i$ to $j$ and the set 
and $\Net_{I'_1,J'_1,I'_2,J'_2}^{ji}$. 
%of networks in $\Net_{I'_1,J'_1,I'_2,J'_2}$ containing a path from $j$ to $i$. 
Moreover, a suitable modification of Theorem~\ref{realiz} remains true for networks 
in $\Net_{I_1,J_1,I_2,J_2}^{ij}$: these networks represent all rational matrix function such that corresponding
component of the matrix does not vanish identically. To see that we use the following construction. Let $v$ be the neighbor of $b_i$ in $N$ and $u$ be the neighbor of $b_j$ in $N$. Add two new white vertices $v'$ and $v''$ and two new black vertices $u'$ and $u''$. Replace edge $(b_i,v)$ by the edges $(b_i,v')$ and $(v',v)$ so that the weight of
the obtained path is equal to the weight of the replaced edge. In a similar way, replace $(u,b_j)$ by $(u,u')$ and 
$(u',b_j)$. Besides, add edges $(v',v'')$ and $(u'',u')$ of weight~1 and two parallel edges $(v'',u'')$, one of weight~1, and the other of weight $-1$. Finally, resolve all the arising intersections with the help of the
network $N_{\id}$. Since the set of functions representable via networks in $\Net_{I_1,J_1,I_2,J_2}^{ij}$ is dense in
the space of all rational matrix functions, 
the 2-parametric family $\{\cdot,\cdot\}_{I_1,J_1,I_2,J_2}$ is defined uniquely already by the fact that $M_N$ is Poisson for any $N\in \Net_{I_1,J_1,I_2,J_2}^{ij}$.
Recall that the boundary measurement map $M_N$ factors through $\FF_N$; clearly, the
same holds for the Grassmannian boundary measurement map $X_N$. Besides, the path reversal map  $R^P$ commutes with the projection $y\:\EE_N\to \FF_N$ and commutes with $X_N$ up to $S_{t^P}$. 
Finally, Poisson brackets satisfying relations~\eqref{psre1}-\eqref{psre4} 
and~\eqref{psre21}-\eqref{psre23} commute with  $S_{t^P}$.
%both the boundary measurement map and the projection $y\:\EE_N\to \FF_N$.
Therefore, Poisson structures $\{\cdot,\cdot\}_{I_1,J_1,I_2,J_2}$ and $\{\cdot,\cdot\}_{I'_1,J'_1,I'_2,J'_2}$ coincide on $LG_k^I(n)\cap LG_k^{I'}(n)$.
If $|I\cap I'|=r<k-1$, we consider a sequence $(I=I^{(0)},I^{(1)},\dots,I^{(k-r)}=I')$ such that $|I^{(t)}\cap I^{(t+1)}|=k-1$ for all $t=0,\dots,k-r-1$ and apply to each pair $( I^{(t)},I^{(t+1)})$ the same reasoning as above.

\endproof

\section{Acknowledgments}

We wish to express gratitude to A.~Postnikov who explained to us the details
of his construction and to V.~Fock, S.~Fomin and N.~Reshetikhin for stimulating discussions. M.~G.~was supported in part by NSF Grants DMS \#0400484 and DMS \#0801204. 
M.~S.~was supported in part by NSF Grants DMS \#0401178 and PHY \#0555346.  A.~V.~was supported in part by ISF Grant \#1032/08.

\end{document}